\documentclass{amsart}

\usepackage[dvipdfmx]{graphicx}
\usepackage{color}

\usepackage{float}
\usepackage{array,booktabs}
\usepackage{subcaption}
\usepackage{txfonts}
\usepackage{amsmath}
\usepackage{subcaption}
\usepackage{bm}
\usepackage{mathrsfs}

\newtheorem{theorem}{Theorem}[section]
\newtheorem{lemma}[theorem]{Lemma}
\newtheorem{proposition}[theorem]{Proposition}
\theoremstyle{definition}
\newtheorem{definition}[theorem]{Definition}
\newtheorem{example}[theorem]{Example}

\newtheorem{cor}[theorem]{Corollary}
\theoremstyle{remark}

\newcommand{\relmiddle}[1]{\mathrel{}\middle#1\mathrel{}}

\newcommand{\acknowledgmentname}{\textbf{Acknowledgements. }}

\DeclareMathOperator{\bip}{Bip}
\DeclareMathOperator{\ehr}{Ehr}
\DeclareMathOperator{\T}{Top}

\newcommand{\conv}{\operatorname{Conv}}

\numberwithin{equation}{section}

\title {Interior polynomial for signed bipartite graphs and the HOMFLY polynomial}
\author{Keiju Kato}
\date{\today}
\email{kato.k.at@m.titech.ac.jp}
\address{Department of Mathematics, Tokyo Institute of Technology, Oh-okayama 2-12-1, Meguro-ku, Tokyo 152-8551, Japan}

\begin{document}

\begin{abstract}
The interior polynomial is an invariant of bipartite graphs, and a part of the HOMFLY polynomial of a special alternating link coincides with the interior polynomial of the Seifert graph of the link. We extend the interior polynomial to signed bipartite graphs, and we show that, in the planar case, it is equal to a part of the HOMFLY polynomial of a naturally associated link. Moreover we obtain a part of the HOMFLY polynomial of any oriented link from the interior polynomial of the Seifert graph. We also establish some other, more basic properties of this new notion. This leads to new identities involving the original interior polynomial.
\end{abstract}
\maketitle

\section{Introduction}
In this paper, we compute a part of the HOMFLY polynomial of knots and links using a new kind of graph theory. Among polynomial invariants of knots, the Alexander polynomial, which is defined homologically, has been well known. The HOMFLY polynomial $P_{L}(v,z) \in \mathbb{Z}[v^{\pm 1},z^{\pm 1}]$ is an oriented link invariant defined by $P_{\text{unknot}}(v,z)=1$ and the skein relation $ v^{-1}P_{D_{+}}-vP_{D_{-}}=zP_{D_0} $, where $D_{+}$, $D_{-}$, $D_0$ are a skein triple (see Figure \ref{fig:D}) \cite{homfly}. It specializes to the Alexander polynomial $\Delta(t)$ and the Jones polynomial $V(t)$ via the substitutions $\Delta(t)=P(1,t^{1/2}-t^{-1/2})$ and $V(t)=P(t,t^{1/2}-t^{-1/2})$, respectively. In this sense, the HOMFLY polynomial contains the Alexander polynomial and the Jones polynomial. Jaeger \cite{jaeger} and Traldi \cite{weighted} computed the HOMFLY polynomial by using graph theory but only for links presentable by certain very special diagrams. Jaeger showed that the Tutte polynomial $T(x,y)$ of a planar graph is equivalent to the HOMFLY polynomial of a certain associated link. This link has a diagram in which two crossings correspond to each edge of the graph. Traldi extended the Tutte polynomial to weighted graphs and showed that in the planar case it is equivalent to the HOMFLY polynomial of an associated link, which is similar to Jaeger's but not necessarily alternating.

\begin{figure}[htbp]
\begin{tabular}{ccc}
\begin{minipage}{0.15\hsize}
\centering
\includegraphics[width=1cm]{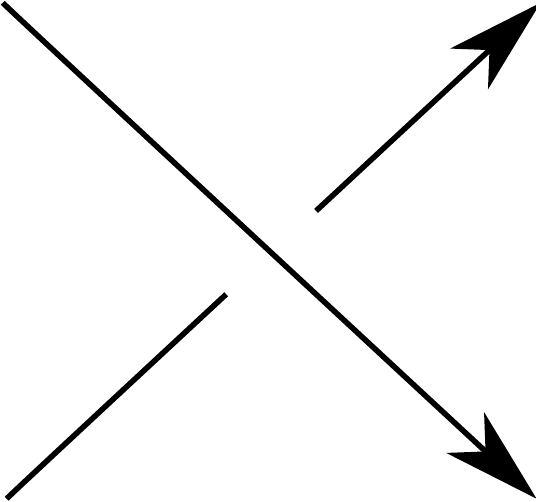}
\\ \vspace{2pt}
$D_{+}$
\end{minipage}
\begin{minipage}{0.15\hsize}
\centering
\includegraphics[width=1cm]{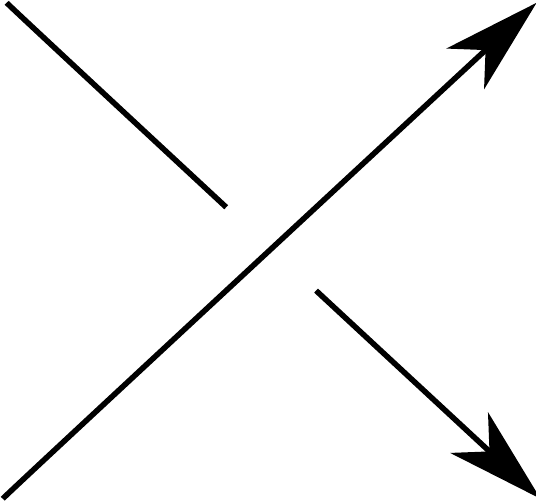}
\\ \vspace{2pt}
$D_{-}$
\end{minipage}
\begin{minipage}{0.15\hsize}
\centering
\includegraphics[width=1cm]{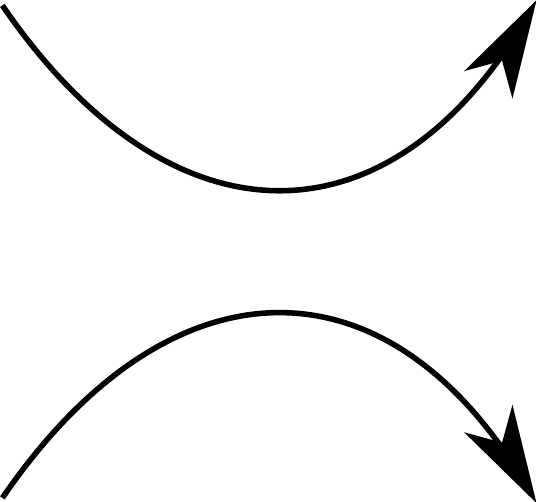}
\\ \vspace{2pt}
$D_{0}$

\end{minipage}
\end{tabular}
\caption{Skein triple.}
\label{fig:D}
\end{figure}

Regarding arbitrary links, Morton \cite{morton} showed that the maximal $z$ exponent in the HOMFLY polynomial of an oriented link diagram $D$ is less than or equal to $c(D)-s(D)+1$, where $c(D)$ is the crossing number of $D$ and $s(D)$ is the number of its Seifert circles. We call the coefficient of $z^{c(D)-s(D)+1}$, which is a polynomial in $v$, the top of the HOMFLY polynomial and denote it by $\T_D(v)$. We note that this depends on $D$ and not just on the link presented by $D$. For any plane graph $G$, Jaeger's theorem implies that $\T_D(v)$ has the same coefficients as $T_G(1/x,1)$. Here $D$ is the diagram defined by Jaeger; see Figure \ref{fig:jaeger} for an example. Our first goal is to use similar ideas to obtain the top of the HOMFLY polynomial but to do it for a special diagram of an arbitrary link. Here ``special'' has its usual meaning: no Seifert circles separates other Seifert circles from each other.

\begin{figure}[htbp]
\centering
\includegraphics[width=2cm]{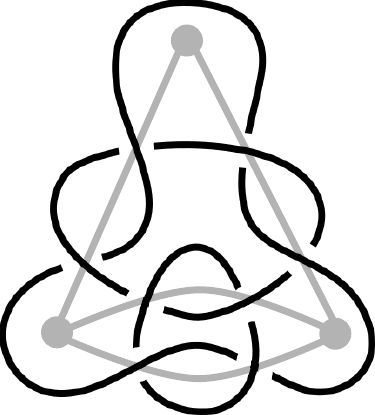}
\caption{Jaeger's link obtained from a plane graph.}
\label{fig:jaeger}
\end{figure}

As a generalization of $T_G(1/x,1)$, K\'alm\'an introduced the interior polynomial \cite{K}. The interior polynomial is a polynomial invariant of hypergraphs or bipartite graphs. Here a hypergraph  $\mathscr{H}=(V,E)$ has a vertex set $V$ and a hyperedge set $E$, where $E$ is a multiset of non-empty subsets of $V$. The interior polynomial is naturally associated to such a structure, but by the main result of \cite{KP}, we may also regard the interior polynomial as an invariant of the natural bipartite graph $\bip\mathscr{H}$ with color classes $E$ and $V$. Moreover, when $G$ is a connected plane bipartite graph, the coefficients of the interior polynomial $I_G$ agree with the coefficients of $\T_{L_G}(v)$, where $L_G$ is the alternating link diagram derived from $G$ by the median construction \cite{KM,KP}. Notice that this time, each edge of $G$ corresponds to just one crossing of $L_G$ and also that $L_G$ is a special alternating diagram with Seifert graph $G$. See Figure \ref{fig:alt} for an example. Therefore we may compute $\T_{L_G}(v)$ of the special alternating diagram $L_G$ by using the interior polynomial. The aim of this paper is to generalize this computation to arbitrary special diagrams.

To this end, we extend the interior polynomial to signed bipartite graphs, that is, bipartite graphs $G$ with a sign $\mathcal{E} \to \{ +1 , -1 \}$, where $\mathcal{E}$ is the set of edges. The signed interior polynomial $I^+_G$ is constructed as an alternating sum of the interior polynomials of the bipartite graphs obtained from $G$ by deleting some negative edges and forgetting the sign. Since the bipartite graph so obtained may be disconnected, we also need to extend the interior polynomial to disconnected bipartite graphs. We build this theory without assuming that $G$ is planar.

Then when $G$ is embedded in the plane, we show that $I^+_G$ agrees with the top of the HOMFLY polynomial of the link diagram $L_G$ obtained from $G$ by replacing positive and negative edges by positive and negative crossings, respectively, see Theorem \ref{thm:main}. We extend Theorem \ref{thm:main} to any oriented link diagram and the Seifert graph. The Seifert graph is the bipartite graph, so we compute the interior polynomial of the Seifert graph. More precisely, the main result of this paper is the following.
\begin{theorem}\label{thm:any}
Let $D$ be an oriented link diagram, and the $G=(V,E,\mathcal{E}_+\cup\mathcal{E}_-)$ be the Seifert graph obtained from $D$, where $E\cup V$ is vertex set which is separated by color and $\mathcal{E}_+\cup \mathcal{E}_-$ is edge set which is separated by sign. Then the top of the HOMFLY polynomial $P_{D}(v,z)$ is equal to
\[
v^{|\mathcal{E}_{+}|-|\mathcal{E}_{-}|-(|E|+|V|)+1}I^{+}_{G}(v^2 ).
\]
\end{theorem}
On the other hand, there exist signed planar bipartite graphs $G=(V,E,\mathcal{E})$ such that $I^+_G(x)=0$. In such cases, the maximal exponent of $z$ in the HOMFLY polynomial of the link diagram $L_G$ is less than $|\mathcal{E}|-(|E|+|V|)+1$, in other words, Morton's in inequality is not sharp. We also find a sufficient condition for $I^+_G=0$, which works in non-planar cases, too. 
\begin{theorem}\label{thm:0}
If a signed bipartite graph $G$ contains an alternating cycle of positive and negative edges, then $I^+_G(x)=0$.
\end{theorem}
To prove this theorem, we use the root polytope and the Ehrhart polynomial. The root polytope is a convex polytope constructed from a bipartite graph. The Ehrhart polynomial of a polytope counts the number of integer points in dilations of the polytope. Moreover, the Ehrhart polynomial of the root polytope of a connected bipartite graph is equivalent to the interior polynomial \cite{KP}. We compute the signed interior polynomial by the Ehrhart polynomial and get the vanishing formula above.

Another important application of Theorem \ref{thm:0} is that it gives an identity for the original interior polynomial $I_G$. This is one possible counterpart of the deletion-contraction relation of the Tutte polynomial, in that it enables one to compute the interior polynomial recursively.
\begin{cor}\label{cor:recursion}
If an unsigned bipartite graph $G$ contains a cycle $\epsilon_1,\delta_1,\epsilon_2,\delta_2,\cdots,\epsilon_n,\delta_n$, then we have
\[
I'_G(x)=\sum_{\emptyset\neq\mathcal{S}\subset\{\epsilon_1,\epsilon_2,\cdots,\epsilon_n\}}(-1)^{|\mathcal{S}|-1}I'_{G\setminus\mathcal{S}}(x).
\]
Here $I'$ refers to our extension of $I$ to not necessarily connected bipartite graphs.
\end{cor}

\noindent
\acknowledgmentname 
I should like to express my gratitude to associate professor Tam\'as K\'alm\'an for constant encouragement and much helpful advice.

\section{The interior polynomial}
\subsection{Preliminaries}
In this section, we recall some definitions and facts about hypergraphs and the interior polynomial that are contained in \cite{K,KM,KP}. A hypergraph is a pair $\mathscr{H}=(V,E)$, where $V$ is a finite set and $E$ is a finite multiset of non-empty subsets of $V$. We order the set $E$ of hyperedges and define the interior polynomial of $\mathscr{H}$ by the activity relation between hyperedges and so called hypertrees \cite{K}. For the set of hypertrees to be non-empty, here we assume that $\mathscr{H}$ is connected. This means that the graph $\bip\mathscr{H}$, defined below, is connected. The interior polynomial does not depend on the order of the hyperedges. It generalizes the evaluation $x^{|V|-1}T_{G}(1/x,1)$ of the classical Tutte polynomial $T_G(x,y)$ of the graph $G=(V,E)$.

We obtain a bipartite graph from the hypergraph $\mathscr{H}=(V,E)$, by letting an edge of the bipartite graph connect a vertex (i.e., an element of $V$) and a hyperedge if the hyperedge contains the vertex. We denote the bipartite graph obtained from the hypergraph $\mathscr{H}$ by $\bip\mathscr{H}=(V,E,\mathcal{E})$. Thus $V$ and $E$ become the color classes of $\bip\mathscr{H}$; in particular, both play the role of vertices. This method gives a two-to-one correspondence from hypergraphs to bipartite graphs. The two hypergraphs corresponding one bipartite graph are called abstract dual. We will denote by $\overline{\mathscr{H}}=(E,V)$ the abstract dual hypergraph of $\mathscr{H}=(V,E)$. Whenever one bipartite graph generates two hypergraphs in this way, the interior polynomials of them are the same \cite{KP}. Therefore we regard the interior polynomial as an invariant of bipartite graphs.

The abstract theory outlined above may be applied in knot theory as follows. Let $L_{G}$ be the special alternating diagram obtained from the plane bipartite graph $G=(V,E,\mathcal{E})$ by replacing each edge by a crossing as shown in Figure \ref{fig:edge}. See also Figure \ref{fig:alt} for an example.
\begin{figure}[H]
\centering
\includegraphics[width=2cm]{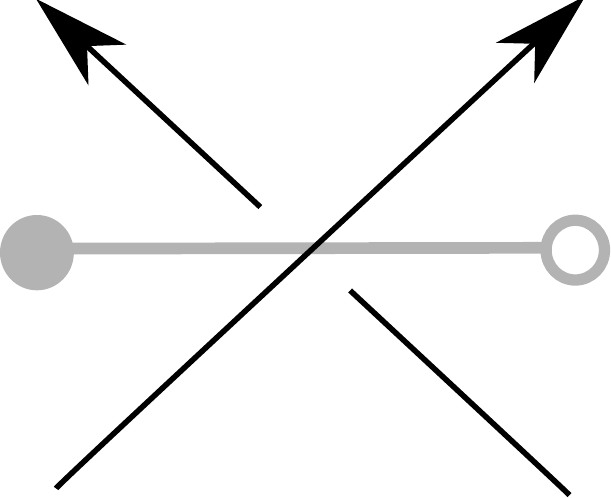}
\caption{A crossing of $L_G$ obtained from an edge of $G$.}
\label{fig:edge}
\end{figure}

\begin{theorem}[\cite{KM,KP}]\label{thm:alt}
For any connected plane bipartite graph $G=(V,E,\mathcal{E})$, the top of the HOMFLY polynomial $P_{L_G}(v,z)$ is equal to $v^{|\mathcal{E}|-(|E|+|V|)+1}I_{G}(v^2)$, where $I_G$ is the interior polynomial of $G$.
\end{theorem}

\begin{example}
Let $G$ be the bipartite graph shown in Figure \ref{fig:alt}, and $L_G$ be the special alternating diagram obtained from $G$. Then $I_G(x)$ and $P_{L_G}(v,z)$ are computed as follows. The coefficients of $\T_{L_G}(v)=1v^3+3v^5+3v^7$ agree with those of $I_G(x)$. 
\begin{figure}[htbp]
\begin{center}
\begin{tabular}{c}
\begin{minipage}{0.2\hsize}
\begin{center}
\includegraphics[width=2.5cm]{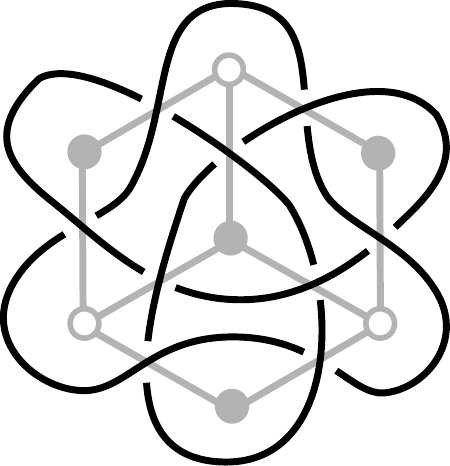}
\end{center}
\end{minipage}

\begin{minipage}{0.7\hsize}
\[
I_{G}(x)=1x^0+3x^1+3x^2.
\]

\[
\begin{array}{llllll}
P_{L_{G}}(v,z)=&+1v^3z^3&+3v^5z^3&+3v^7z^3   &                       \\     
           &        &+3v^5z  &+4v^7z     &-4v^9z                    \\
           &        &        &+2v^7z^{-1}&-3v^9z^{-1}&+1v^{11}z^{-1}.\\
\end{array}
\]
\end{minipage}
\end{tabular}
\caption{A special alternating link diagram and the HOMFLY polynomial.}
\label{fig:alt}
\end{center}
\end{figure}
\end{example}

Before we introduce the interior polynomial of signed bipartite graphs, we introduce the interior polynomial of disconnected bipartite graphs.

\begin{definition}
Let $k(G)$ be the number of components of $G$ and $G_i$ ($1 \le i \le k(G)$) be the connected components of $G$. Then we let
\[
I'_{G}\left(x \right)=\left(1-x \right)^{k(G)-1}\prod_{i=1}^{k(G)} I_{G_i}\left(x \right).
\]
\end{definition}
With this, the next lemma is obvious.
\begin{lemma}\label{lem:disint}
Let $G_1$ and $G_2$ be bipartite graphs and $G_1 \cup G_2$ be the disjoint union of $G_1$ and $G_2$. Then
\[
I'_{G_1 \cup G_2}(x)=(1-x)I'_{G_1}(x)I'_{G_2}(x).
\]
\end{lemma}
Theorem \ref{thm:alt} extends as follows. Note that we are still in the unsigned case.
\begin{theorem}\label{thm:bip}
For any plane bipartite graph $G=(V,E,\mathcal{E})$, the top of the HOMFLY polynomial $P_{L_G}(v,z)$ is equal to
\[
v^{|\mathcal{E}|-(|E|+|V|)+1}I'_{G}(v^2).
\]
\end{theorem}
\begin{proof}
The proof is by induction on the component number $k(G)$.
When $k(G)=1$, the statement holds by Theorem \ref{thm:alt}. Suppose that the theorem holds when $k(G) < m$ and let $G$ have $k(G)=m$ components. Let us take non-empty bipartite graphs $G_1=(V_1,E_1,\,mathcal{E}_1),G_2=(V_2,E_2,\mathcal{E}_2)$ such that $G=G_1\cup G_2$. Recall that the HOMFLY polynomial satisfies
\[
P_{L_{G_1 \cup G_2}}(v,z)=\frac{v^{-1}-v}{z}P_{L_{G_1}}(v,z)P_{L_{G_2}}(v,z).
\]
Paying attention to Morton's bound $n(D)$, if we have $n(L_{G_1})=c(L_{G_1})-s(L_{G_1})+1$ and $n(L_{G_2})=c(L_{G_2})-s(L_{G_2})+1$, then $n(L_{G_1} \cup L_{G_2})=c(L_{G_1} \cup L_{G_2})-s(L_{G_1} \cup L_{G_2})+1=c(L_{G_1})+c(L_{G_2})-s(L_{G_1}) -s(L_{G_2})+1=n(G_1)+n(G_2)-1$. Thus
\begin{eqnarray*}
\T_{L_{G_1 \cup G_2}}(v)
&=&(v^{-1}-v)\T_{L_{G_1}}(v)\T_{L_{G_2}}(v) \\
&=&v^{-1}(1-v^2)\T_{L_{G_1}}(v)\T_{L_{G_2}}(v) \\
&=&v^{-1}(1-v^2)v^{|\mathcal{E}_1|-(|E_1|+|V_1|)+1}I'_{G_1}(v^2)v^{|\mathcal{E}_2|-(|V_2|+|E_2|)+1}I'_{G_2}(v^2) \\
&=&v^{|\mathcal{E}_1|+|\mathcal{E}_2|-(|E_1|+|V_1|+|E_2|+|V_2|)+1}(1-v^2)I'_{G_1}(v^2)I'_{G_2}(v^2)   \\
&=&v^{|\mathcal{E}_1|+|\mathcal{E}_2|-(|E_1|+|E_2|+|V_1|+|V_2|)+1}I'_{G_1 \cup G_2}(v^2).
\end{eqnarray*}
Therefore the theorem holds when $k(G)=m$.
\end{proof}

\subsection{A signed version and the HOMFLY polynomial}
Let $G=(V,E,\mathcal{E})$ be a signed bipartite graph with sign $\mathcal{E} \to \{ +1 , -1 \}$ and $\mathcal{E}_{-}(G)$ (resp.\ $\mathcal{E}_{+}(G)$) be the set of negative (resp.\ positive) edges of $G$.  
Let $\mathcal{S}$ be a subset of $\mathcal{E}_-(G)$. The unsigned bipartite graph $G \setminus \mathcal{S}$ is obtained from G by deleting all edges of $\mathcal{S}$ and forgetting the signs of the remaining edges. So we may compute the interior polynomial of $G \setminus \mathcal{S}$. We will construct the interior polynomial of signed bipartite graphs as follows.
\begin{definition}
Let $G=(V,E,\mathcal{E})$ be a signed bipartite graph. We let
\[
I^{+}_{G}\left(x \right)=\sum_{\mathcal{S} \subseteq \mathcal{E}_{-}(G)}(-1)^{|\mathcal{S}|}I'_{G \setminus \mathcal{S}}(x).
\]
\end{definition}
We call $I^+_G(x)$ the signed interior polynomial.
\begin{example}
Let $G$ be the signed bipartite graph shown in Figure \ref{fig:signbip}. We compute the signed interior polynomial as Table \ref{tab:1}. All (unsigned) interior polynomials that occur in the computations are obtained easily from the definition in \cite{K}. Note that in the last row, $G\setminus \mathcal{S}$ has two components with $I(x)=1+x+x^2$ for the hexagon and $I(x)=1$ for the isolated point.
\begin{figure}[H]
\centering
\includegraphics[width=2cm]{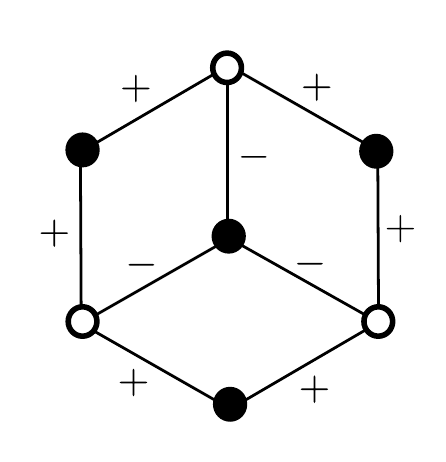}
\caption{A signed bipartite graph.}
\label{fig:signbip}
\end{figure}

\begin{table}[htbp]
\centering
\caption{A computation of signed interior polynomial $I^+_G(x)$.}\label{tab:1}
\begin{tabular}{lllllllll}
& &  & & & $(-1)^{|\mathcal{S}|}I'_{G \setminus \mathcal{S}}$　　\\
\hline \vspace{4pt}
           
           &\includegraphics[width=1cm]{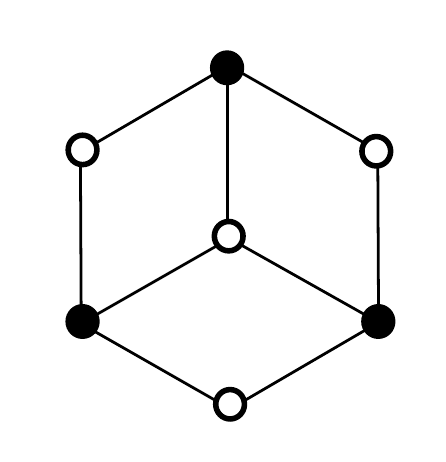}
           &                                             &
&          &\raisebox{+8pt}{\hspace{3pt}$1x^{0}+3x^{1}+3x^{2}$}   
           &\raisebox{+8pt}{$\times 1$}\\    
            \includegraphics[width=1cm]{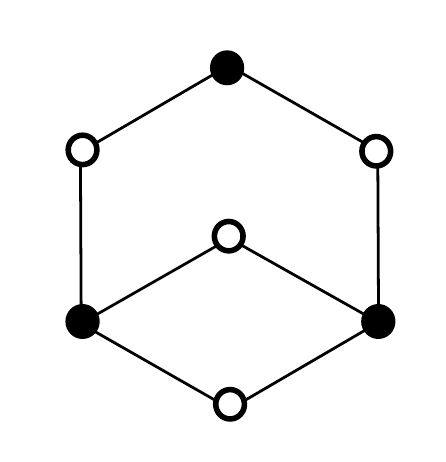}
           & \includegraphics[width=1cm]{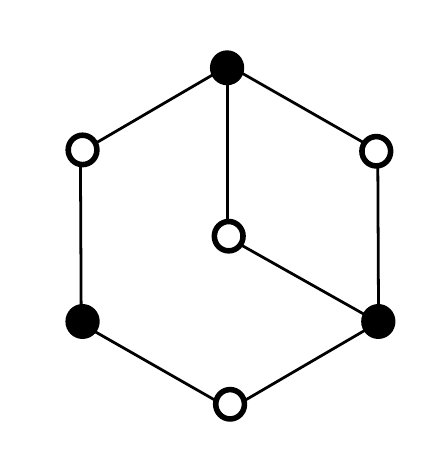}
           & \includegraphics[width=1cm]{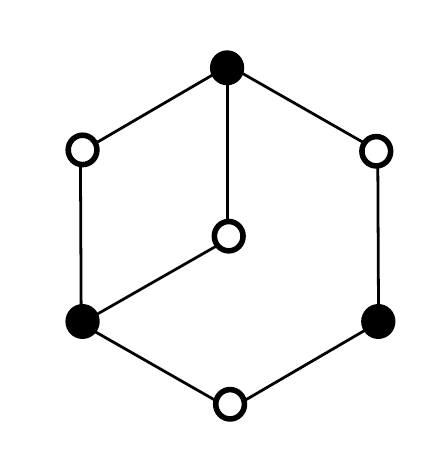}& 
&\raisebox{+8pt}{$-$\hspace{-10pt}}&\raisebox{+8pt}{$(1x^{0}+2x^{1}+2x^{2})$}
           & \raisebox{+8pt}{$\times 3$}\\ 
            \includegraphics[width=1cm]{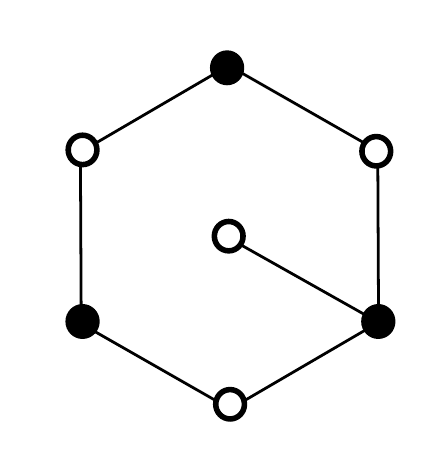}
           & \includegraphics[width=1cm]{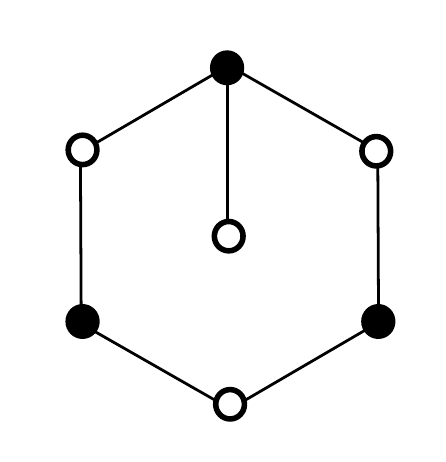}
           & \includegraphics[width=1cm]{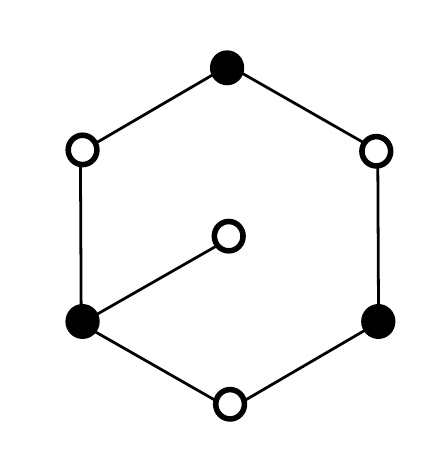}&
&\raisebox{+8pt}&\raisebox{+8pt}{\hspace{3pt}$1x^{0}+1x^{1}+1x^{2}$}         
           & \raisebox{+8pt}{$\times 3$}\\ 
             
           & \includegraphics[width=1cm]{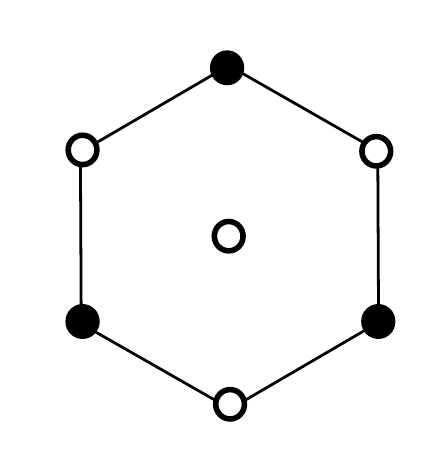} 
           &                                            &
&\raisebox{+8pt}{$-$}&\raisebox{+8pt}{$(1x^{0}+0x^{1}+0x^{2}-1x^{3})$}　& \raisebox{+8pt}{$\times 1$}　　\\ 
\hline \\
 & &$I^{+}_{G}(x)$&$=$& &\hspace{3pt}$0x^{0}+0x^{1}+0x^{2}+1x^{3}$
\end{tabular}
\end{table}
\end{example}

Before we show that the signed interior polynomial of a plane bipartite graph $G$ is equivalent to the top of the HOMFLY polynomial of $L_G$, we show the following lemma for (not necessarily plane) bipartite graphs. Let $G$ be a signed bipartite graph and let $\epsilon$ be one of the negative edges in $G$. The bipartite graph $G \setminus \epsilon$ is obtained from $G$ by deleting $\epsilon$ and $G+\epsilon$ is obtained from $G$ by replacing the negative edge $\epsilon$ by a positive edge (see Figure \ref{fig:replace}).

\begin{figure}[htbp]
\begin{tabular}{ccc}
\begin{minipage}{0.33\hsize}
\begin{center}
\includegraphics[width=3.5cm]{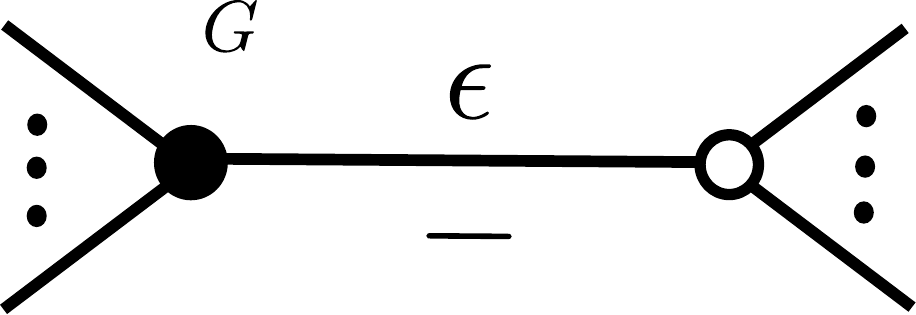}
\end{center}
\end{minipage}
\begin{minipage}{0.33\hsize}
\begin{center}
\includegraphics[width=3.5cm]{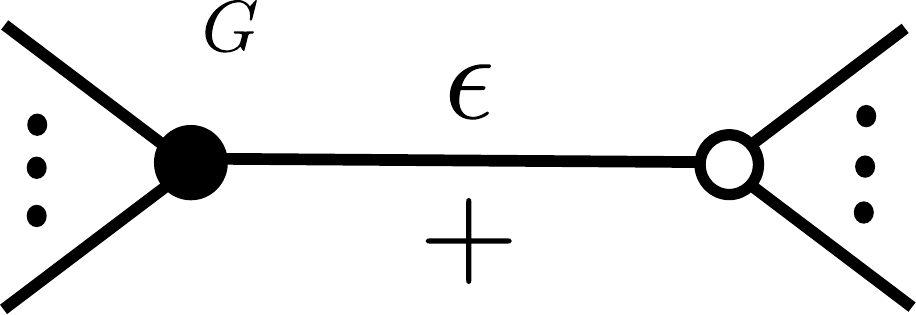}
\end{center}
\end{minipage}
\begin{minipage}{0.33\hsize}
\begin{center}
\includegraphics[width=3.5cm]{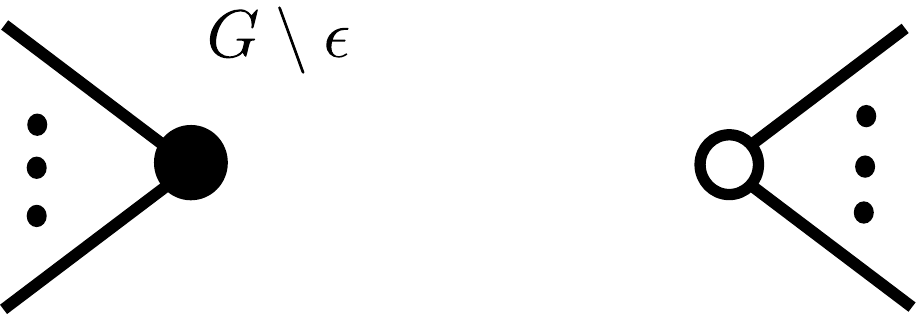}
\end{center}
\end{minipage}
\end{tabular}
\caption{A version of the skein triple for signed bipartite graph.}\label{fig:replace}
\end{figure}

\begin{lemma}\label{lem:sumint}
Let $G$ be a signed bipartite graph and let $\epsilon$ be one of the negative edges in $G$. Then we have $I^+_G(x)=I^+_{G+\epsilon}(x)-I^+_{G \setminus \epsilon}(x)$.

\end{lemma}
\begin{proof}
By the definition of signed interior polynomial,
\begin{eqnarray*}
I^+_G(x)&=&\sum_{\mathcal{S} \subseteq \mathcal{E}_{-}(G)}(-1)^{|\mathcal{S}|}I'_{G \setminus \mathcal{S}}(x) \\
         &=&\sum_{\mathcal{S} \subseteq \mathcal{E}_{-}(G) \atop \epsilon \notin \mathcal{S}}(-1)^{|\mathcal{S}|}I'_{G \setminus \mathcal{S}}(x) \,\,\, +  \sum_{ \mathcal{S} \subseteq \mathcal{E}_{-}(G) \atop \epsilon \in \mathcal{S}}(-1)^{|\mathcal{S}|}I'_{G \setminus \mathcal{S}}(x)  \\
         &=&\sum_{ \mathcal{S} \subseteq \mathcal{E}_{-}(G+\epsilon)}(-1)^{|\mathcal{S}|}I'_{(G+\epsilon) \setminus \mathcal{S}}(x) \,\,\, -  \sum_{(\mathcal{S} \setminus \epsilon) \subseteq \mathcal{E}_{-}(G \setminus \epsilon)}(-1)^{|\mathcal{S} \setminus \epsilon|}I'_{(G \setminus \epsilon) \setminus (\mathcal{S} \setminus \epsilon)}(x)  \\
         &=&I^+_{G+\epsilon}(x)-I^+_{G \setminus \epsilon}(x).
\end{eqnarray*}
This completes the proof.
\end{proof}

For any signed bipartite graph $G$, the link diagram $L_G$ is obtained from $G$ by replacing positive and negative edges by positive and negative crossings, respectively, as shown in Figure \ref{fig:posneg}.
\begin{figure}[htbp]
\begin{tabular}{ccc}
\begin{minipage}{0.4\hsize}
\begin{center}
\includegraphics[width=2cm]{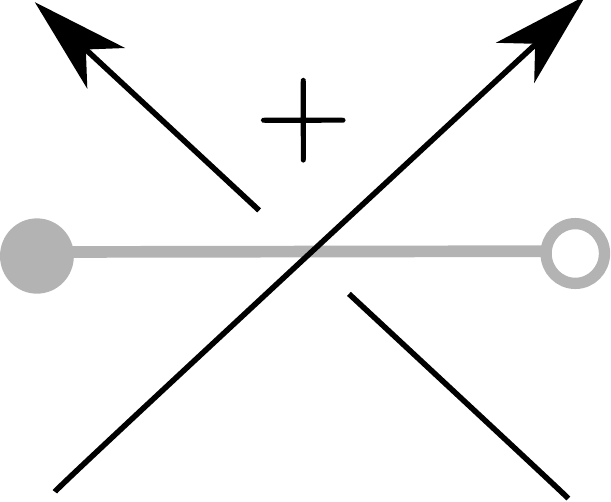}
\end{center}
\end{minipage}
\begin{minipage}{0.4\hsize}
\begin{center}
\includegraphics[width=2cm]{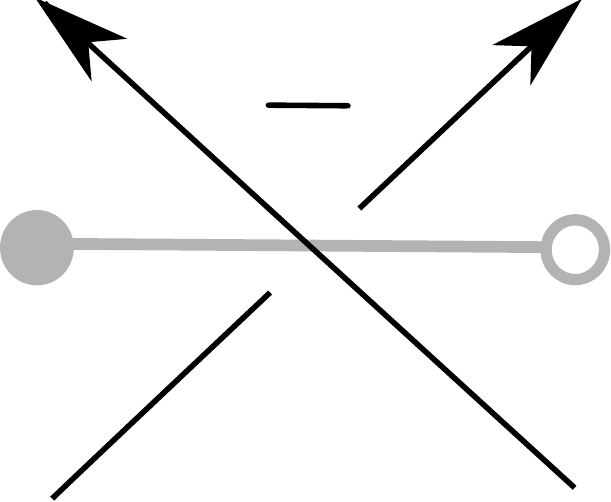}
\end{center}
\end{minipage}
\end{tabular}
\caption{The positive crossing and the negative crossing corresponding to a positive edge and a negative edge, respectively.}
\label{fig:posneg}
\end{figure}

Therefore we extend Theorem \ref{thm:alt} to signed bipartite graph.
\begin{theorem}\label{thm:main}
Let $G$ be a signed plane bipartite graph with vertex set $E\cup V$ separated by color and edge set $\mathcal{E}_+\cup \mathcal{E}_-$ separated by sign. Then the top of the HOMFLY polynomial $P_{L_G}(v,z)$ is equal to
\[
v^{|\mathcal{E}_{+}|-|\mathcal{E}_{-}|-(|E|+|V|)+1}I^{+}_{G}(v^2 ).
\]
\end{theorem}
\begin{proof}
We proceed by induction on the number of negative edges $|\mathcal{E}_{-}|$.
When $|\mathcal{E}_{-}|=0$, we have $I^{+}_{G}(x)=I'_{G}(x)$. Thus the statement holds by Theorem \ref{thm:bip}.

When $|\mathcal{E}_{-}|<m$, we suppose that the theorem holds and let a bipartite graph $G$ have $m$ negative edges. We take a negative edge $\epsilon$ in $G$. From Lemma \ref{lem:sumint}, we have $I^+_G(x)=I^+_{G+\epsilon}(x)-I^+_{G \setminus \epsilon}(x)$. Now by the induction hypothesis applied to $G+\epsilon$ and $G\setminus\varepsilon$, and the fact that $L_{G+\epsilon},L_G, L_{G\setminus\epsilon}$ form a skein triple, we have

\begin{eqnarray*}
I^+_G(v^2)&=&I^+_{G+\epsilon}(v^2)-I^+_{G \setminus \epsilon}(v^2) \\
         &=&\frac{1}{v^{(|\mathcal{E}_{+}|+1)-(|\mathcal{E}_{-}|-1)-(|E|+|V|)+1}}\T_{L_{G+\epsilon}}(v)-\frac{1}{v^{|\mathcal{E}_{+}|-(|\mathcal{E}_{-}|-1)-(|E|+|V|)+1}}\T_{L_{G \setminus \epsilon}}(v)\\
         &=&\frac{1}{v^{|\mathcal{E}_{+}|-|\mathcal{E}_{-}|-(|E|+|V|)+1}}\left(v^{-2}\T_{L_{G+\epsilon}}(v)-v^{-1}\T_{L_{G \setminus \epsilon}}(v)\right)  \\
         &=&\frac{1}{v^{|\mathcal{E}_{+}|-|\mathcal{E}_{-}|-(|E|+|V|)+1}}\T_{L_{G}}(v).
\end{eqnarray*}
For the last step note that Morton's bound is the same for $G+\epsilon$ and $G$, and it is one less than that for $G\setminus \epsilon$. Hence $\T_{L_G}(v)=v^{|\mathcal{E}_{+}|-|\mathcal{E}_{-}|-(|E|+|V|)+1}I^+_G(v^2)$. Therefore the theorem holds when $|\mathcal{E}_{-}|=m$.
\end{proof}
\begin{example}
Let $G$ be the signed bipartite graph shown in Figure \ref{fig:signbip}. Then the link diagram $L_G$ obtained from $G$ is as shown in Figure \ref{fig:nonalt}. We compute the HOMFLY polynomial of $L_G$ as follows. The coefficients of $\T_{L_G}(v)=1v^3$ agree with those of $I^+_G(x)$.
\begin{figure}[htbp]
\begin{tabular}{cc}
\begin{minipage}{0.2\hsize}
\begin{center}
\includegraphics[width=2.5cm]{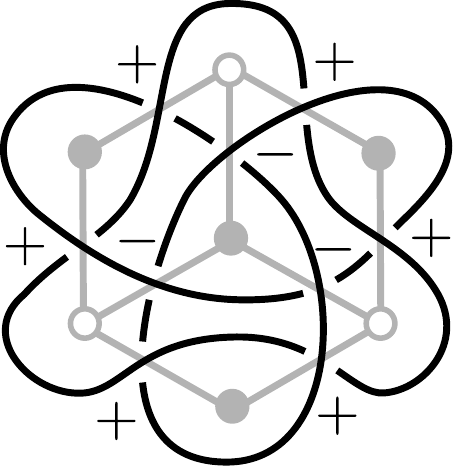}
\end{center}
\end{minipage}
\begin{minipage}{0.7\hsize}
\[
\begin{array}{llll}
P_{L_G}(v,z)=&           &+1v^3z^3                    \\     
         &           &+4v^3z      &-1v^5z         \\
         &-1vz^{-1}  &+3v^3z^{-1} &-2v^5z^{-1}.
\end{array}
\]
\end{minipage}
\end{tabular}
\caption{A non-alternating but special link diagram and the HOMFLY polynomial.}
\label{fig:nonalt}
\end{figure}
\end{example}
If the diagram $D$ is homogeneous \cite{C} (for example, if it is special alternating), then Morton's bound is sharp: the maximal exponent of $z$ in the HOMFLY polynomial is equal to $c(D)-s(D)+1$. For the special alternating link diagram $L_G$ obtained from a plane bipartite graph $G$, the maximal exponent of $z$ is $c(L_G)-s(L_G)+1=|\mathcal{E}|-(|E|+|V|)+1$. In signed cases, Morton's bound may not be sharp and then $I^+_G=0$. We find a sufficient condition for this phenomenon in section \ref{sec:ehr}.

For any oriented diagram, We prove same statement. First, we introduce the Seifert graph of an oriented link diagram $D$. A Seifert circle is a simple closed curve which results when we smooth every crossing of an oriented link diagram in the orientation preserving way. There are two kinds of Seifert circles. We say that a Seifert circle is Type I if it does not contain any other Seifert circles, otherwise it is Type I\hspace{-.1em}I. If the diagram has no Seifert circles of Type I\hspace{-.1em}I, the diagram is a special diagram. The vertex set of the Seifert graph consists of Seifert circles and the edge set consists of crossings with signs. A Seifert graph is always a bipartite graph. If $G$ is the Seifert graph of a special link diagram $D$, then $L_G=D$.

Next, we prove one property of the signed interior polynomial related to block sum.
\begin{theorem}\label{thm:murasugisum}
Let $G_1=(V_1,E_1,\mathcal{E}_1)$ and $G_2=(V_2,E_2,\mathcal{E}_2)$ be bipartite graphs such that $G_1 \cap G_2=\{ v\}$. Let $G_1*G_2$ be the graph obtained by identifying this one vertex. Then
\[
I^+_{G_1 * G_2}(x)=I^+_{G_1}(x) I^+_{G_2}(x).
\]
\end{theorem}
\begin{proof}
First, we prove the statement of the theorem for disconnected bipartite graphs and the interior polynomial $I'$. Consider the possibly disconnected but unsigned bipartite graphs $G_1=(V_1,E_1,\mathcal{E}_1)$ and $G_2=(V_2,E_2,\mathcal{E}_2)$, which have $k(G_1)=n$ and $k(G_2)=m$ components, respectively. We take the disjoint connected bipartite graphs $G_1^1,\ldots,G_1^n$ such that $G_1=G_1^1\cup \cdots \cup G_1^n$. In the same way, we take the disjoint connected bipartite graphs $G_2^1,\ldots,G_2^m$ such that $G_2=G_2^1\cup \cdots \cup G_2^m$. Without loss of generality, we can suppose that $G_1 \cap G_2=G_1^1\cap G_2^1=\{ v\}$. From \cite[Theorem 6.7]{K}, we have $I_{G_1^1* G_2^1}(x)=I_{G_1^1}(x)I_{G_2^1}(x)$. From the definition of $I'$, we have 
\begin{eqnarray*}
I'_{G_1}(x)I'_{G_2}(x)
&=&(1-x)^{n-1}I_{G_1^1}(x)I_{G_1^2}(x)\cdots I_{G_1^n}(x) \cdot(1-x)^{m-1}I_{G_2^1}(x) I_{G_2^2}(x)\cdots I_{G_2^m}(x)\\
&=&(1-x)^{n+m-2}I_{G_1^1*G_2^1}(x)I_{G_1^2}(x)\cdots I_{G_1^n}(x) I_{G_2^2}(x)\cdots I_{G_2^m}(x)\\
&=&I'_{G_1 * G_2}(x).
\end{eqnarray*}
The theorem follows easily from this and from the definiton of the signed interior polynomial.
\end{proof}
In \cite{MP}, the block sum in graph theory is related to the $*$-product (or Murasugi-sum) in knot theory and Murasugi and Pryzytycki showed that $\T_{{D}*{D'}}(v)=\T_{D}(v)\T_{D'}(v)$. Now we are in a position to prove our main theorem.

\begin{proof}[proof of Theorem \ref{thm:any}]
We may assume that $G$ is connected. If there are no Seifert circles of Type I\hspace{-.1em}I in $D$, $D$ is special alternating diagram which is obtained from $G$. It is clear by Theorem \ref{thm:main} that this statement holds. 

So we assume that there are Seifert circles of Type I\hspace{-.1em}I. Let $G=G_1*\cdots*G_n$, $n\ge2$, be the block sum decomposition of $G$. And let the link diagrams $D_1,\ldots,D_n$ be associated to $G_1,\ldots,G_n$. Since $D=D_1*\cdots *D_n$, we have $\T_{D}(v)=\T_{D_1}(v)\cdots\T_{D_n}(v)$. Moreover, since the diagrams $D_1,\ldots,D_n$ contain no Seifert circles of Type I\hspace{-.1em}I and $D_1=L_{G_1},\ldots,D_n=L_{G_n}$, from Theorem \ref{thm:main}, we have $\T_{D_1}(v)=v^{i_1}I^+_{G_1}(v^2),\ldots,\T_{D_n}(v)=v^{i_n}I^+_{G_n}(v^2)$, where the integers $i_1,\ldots,i_n$ are the suitable exponents from Theorem \ref{thm:main}. Write $G_k=(V_k,E_k,\mathcal{E}_+^k,\mathcal{E}_-^k)$, then we get $i_k=|\mathcal{E}_{+}^k|-|\mathcal{E}_{-}^k|-(|V_k|+|E_k|)+1$.  Therefore, using the formula of $*$-product and \ref{thm:murasugisum}
\begin{eqnarray*}
\T_{D}(v)
&=&\T_{D_1*\cdots*D_n}(v) \\
&=&\T_{D_1}(v)\cdots\T_{D_n}(v)\\
&=&v^{|\mathcal{E}_{+}^1|-|\mathcal{E}_{-}^1|-(|V_1|+|E_1|)+1}I^{+}_{G_1}\left(v^2 \right)\cdots v^{|\mathcal{E}_{+}^n|-|\mathcal{E}_{-}^n|-(|V_n|+|E_n|)+1}I^{+}_{G_n}\left(v^2 \right) \\
&=&v^{(|\mathcal{E}_{+}^1|+\cdots+|\mathcal{E}_{+}^n|)-(|\mathcal{E}_{-}^1|+\cdots+|\mathcal{E}_{-}^n|)-((|V_1|+\cdots+|V_n|)+(|E_1|+\cdots+|E_n|))+n}I^{+}_{G_1} \left(v^2 \right)\cdots I^{+}_{G_n}\left(v^2 \right)\\
&=&v^{|\mathcal{E}_{+}|-|\mathcal{E}_{-}|-(|V|+|E|+(n-1))+n}I^{+}_{G}\left(v^2 \right)\\
&=&v^{|\mathcal{E}_{+}|-|\mathcal{E}_{-}|-(|V|+|E|)+1}I^{+}_{G}\left(v^2 \right).
\end{eqnarray*}
This completes the proof.
\end{proof}

\subsection{Properties of the signed interior polynomial}
We show some properties of the interior polynomial for abstract signed bipartite graphs. Some of them are the same as for the original interior polynomial in \cite{K}.
\begin{proposition}
The signed interior polynomial of a signed bipartite graph has $1$ or $0$ for constant term depending on whether negative edges are present.
\end{proposition}
\begin{proof}
When $|\mathcal{E}_{-}|=0$, we have $I^+_G=I'_G$. By \cite[Proposition 6.2]{K} and the definition of $I'(x)$, the constant term of $I'_G(x)$ is $1$.

When $|\mathcal{E}_{-}|=m>0$, then $\displaystyle I^{+}_{G}\left(x \right)=\sum_{\mathcal{S} \subseteq \mathcal{E}_{-}(G)}(-1)^{|\mathcal{S}|}I'_{G \setminus \mathcal{S}}\left(x \right)$. Here $I'_{G \setminus \mathcal{S}}$ always has $1$ for constant term. Then we compute the constant term as follows:
\[
\binom{m}{0} - \binom{m}{1} \pm \cdots +(-1)^{m} \binom{m}{m} =(1-1)^m=0.
\]
This completes the proof.
\end{proof}
\begin{theorem}\label{thm:dis}
Let $G_1$ and $G_2$ be signed bipartite graphs and $G_1 \cup G_2$ be the disjoint union of $G_1$ and $G_2$. Then
\[
I^+_{G_1 \cup G_2}(x)=(1-x)I^+_{G_1}(x)I^+_{G_2}(x).
\]
\end{theorem}
\begin{proof}
Obvious from the definition of the signed interior polynomial and Lemma \ref{lem:disint}
\end{proof}
Having multiple edges in a bipartite graph does not affect the induced hypergraphs and hence, in the unsigned case, does not affect the interior polynomial. In the signed case, multiple edges lead to some interesting phenomena. The third formula of the next proposition is a special case of the vanishing formula by the alternating cycle (see section \ref{sec:alt}).

\begin{proposition}\label{thm:nugatory}
From the signed bipartite graph $G$, construct another bipartite graph $G'$ by adding a new vertex that is connected to just one old vertex with $m\ge1$ new edges. Then
\[
\begin{array}{lllll}
I^+_{G'}(x)&=             &\hspace{-8pt}I^+_{G}(x)&(\mbox{if all the new edges are positive}),\\
I^+_{G'}(x)&=(-1)^{m+1}x  &\hspace{-8pt}I^+_{G}(x)&(\mbox{if all the new edges are negative}),\\
I^+_{G'}(x)&=             &\hspace{-8pt}0         &(\mbox{otherwise}).
\end{array}
\]
\end{proposition}
\begin{proof}
For any non-negative integers $i,j$, let $G_{(i,j)}$ be the signed bipartite graph with only two vertices, $n$ positive edges and $j$ negative edges. We compute the signed interior polynomial of $G_{(i,j)}$. When $j=0$, since the interior polynomial of the unsigned bipartite graph obtained from $G_{(i,0)}$ by forgetting signs is $1$, we have $I^+_{G_{(i,0)}}(x)=I_{G_{(i,0)}}=1$. When $i=0$, the signed bipartite graph $G_{(0,j)}$ has $j$ negative edges $\{ \epsilon_1, \epsilon_2, \cdots ,\epsilon_j \}$. Then the unsigned bipartite graph $G_{(0,j)} \setminus \{  \epsilon_1, \epsilon_2, \cdots ,\epsilon_j \}$ has two isolated vertices and the interior polynomial of it is $1-x$. Moreover the other bipartite graphs $G_{(i,j)} \setminus \{  \epsilon_{i_1}, \epsilon_{i_2}, \cdots ,\epsilon_{i_k} \}$ are connected and the interior polynomial of them is $1$. By the definition of the signed interior polynomial, we have 
\begin{eqnarray*}
I^{+}_{G_{(0,j)}}(x)
&=&I^+_{G_{(0,j)} \setminus \emptyset}(x)-I^+_{G_{(0,j)}\setminus \{ \epsilon_{1} \}}(x)-I^+_{G_{(0,j)} \setminus \{ \epsilon_{2} \}}(x)-\cdots -I^+_{G_{(0,j)} \setminus \{ \epsilon_{j} \}}(x)\\
&&\hspace{25pt}+I^+_{G_{(0,j)} \setminus \{ \epsilon_1,\epsilon_2 \}}(x)+\cdots+
I^+_{G_{(0,j)} \setminus \{ \epsilon_{j-1},\epsilon_j \}}(x)\pm \cdots +(-1)^j I^+_{G_{(i,j)} \setminus \{  \epsilon_1, \epsilon_2, \cdots ,\epsilon_j \}}(x) \\
&=&1-\binom{j}{1}+\binom{j}{2}\pm \cdots +(-1)^{j-1}\binom{j}{j-1} +{(-1)^j} \binom{j}{j}(1-x) \\
&=&(-1)^{j+1}x.
\end{eqnarray*}
When $i\ne0$ and $j\ne0$, since $G_{(i,j)} \setminus \{  \epsilon_{i_1}, \epsilon_{i_2}, \cdots ,\epsilon_{i_k} \}$ is always connected, its interior polynomial is $1$. Then all terms cancel and we get $I^+_{G_{(i,j)}}(x)=0$.

By Theorem \ref{thm:murasugisum}, we have $I^+_{G'}(x)=I^+_{G*G_{(i,j)}}(x)=I^+_G(x) I^+_{G_{(i,j)}}(x)$. By the above, we obtain the statement of the proposition.
\end{proof}

\begin{theorem}\label{thm:product}
Let $G_1=(V_1,E_1,\mathcal{E}_1)$ and $G_2=(V_2,E_2,\mathcal{E}_2)$ be signed bipartite graphs. Define the signed bipartite graph $G=(V_1\cup V_2,E_1\cup E_2,\mathcal{E}_1 \cup \mathcal{E}_2 \cup \{ \epsilon\})$, where $\epsilon$ is a new edge connecting two vertices which are chosen opportunely from each bipartite graph $G_1$ and $G_2$. Then
\[
\begin{array}{lllll}
I^+_{G}(x)&=&  &\hspace{-8pt}I^+_{G_1}(x) I^+_{G_2}(x)&(\epsilon\colon\mbox{positive}),\\
I^+_{G}(x)&=&x&\hspace{-8pt}I^+_{G_1}(x) I^+_{G_2}(x)&(\epsilon\colon\mbox{negative}).
\end{array}
\]
\end{theorem}
\begin{proof}
Immediate from Theorem \ref{thm:murasugisum} and a special case of Proposition \ref{thm:nugatory}.
\end{proof}

If $v$ is a vertex in the bipartite graph $G$, then let $G-v$ be the bipartite graph obtained from $G$ by removing $v$ and all edges are incident to it. Moreover let $G/v$ be the bipartite graph obtained from $G-v$ by identifying vertices contained in edges which are adjacent to $v$.
\begin{proposition}
Let $G$ be a signed bipartite graph that contains a vertex $v$ which only $k$ positive edges and $l$ negative edges are incident to so that $G-v$ has $k+l$ connected components. Then $I^+_G(x)=x^lI^+_{G/v}(x)$.
\end{proposition}
\begin{proof}
Immediate from Theorem \ref{thm:murasugisum} and a special case of Proposition \ref{thm:nugatory}.
\end{proof}
The following theorem is a signed version of the classical deletion-contraction formula. The formula covers the case of a bridge but dose not cover the case of a loop.
\begin{theorem}\label{thm:delcont}
Let $G$ be a signed bipartite graph with a vertex $v$ which is incident to only two edges $\epsilon_1,\epsilon_2$. Assume that $\epsilon_1,\epsilon_2$ are not parallel. Then
\[
\begin{array}{llllll}
I^+_G(x)&=&     &\hspace{-8pt}I^+_{G-v}(x)+&\hspace{-8pt}x I^+_{G/v}(x) &(\epsilon_1,\epsilon_2\colon\mbox{positive}), \\

I^+_G(x)&=&-x &\hspace{-8pt}I^+_{G-v}(x)+&\hspace{-8pt}x I^+_{G/v}(x) &(\epsilon_1,\epsilon_2\colon\mbox{negative}), \\
I^+_G(x)&=&     &  \hspace{-8pt}        &\hspace{-8pt}x I^+_{G/v}(x) &(\epsilon_1\colon\mbox{positive},\, \epsilon_2\colon\mbox{negative}). \\
\end{array}
\]
\end{theorem}
\begin{proof}
First, for any unsigned bipartite graph, we show that $I'_G(x)=I'_{G-v}(x)+x I'_{G/v}(x)$. When $G-v$ is connected, $I_G(x)=I_{G-v}(x)+x I_{G/v}(x)$ from \cite[Proposition 6.14]{K}. Thus $I'_G(x)=I'_{G-v}(x)+x I'_{G/v}(x)$ holds. When $G-v$ is disconnected, We take the connected component $G_1$, $G_2$ such that $G-v=G_1\cup G_2$. From Lemma \ref{lem:disint}, we have $I'_{G-v}(x)=(1-x)I'_{G_1}(x)I'_{G_2}(x)$. On the other hand, the bipartite graph $G/v$ is $G_1*G_2$, which is obtained from $G_1$ and $G_2$ by identifying vertices contained in edges which connect with $v$ in $G$. From an unsigned case of Theorem \ref{thm:murasugisum}, we have $I'_{G/v}(x)=I'_{G_1*G_2}=I'_{G_1}(x)I'_{G_2}(x)$. From an unsigned case of Theorem \ref{thm:product}, we have $I'_G(x)=I'_{G_1}(x)I'_{G_2}(x)$. Therefore, we have
\[
\begin{array}{llllll}
I'_{G-v}(x)+x I'_{G/v}(x)&=&(1-x)I'_{G_1}(x)I'_{G_2}(x)+x I'_{G_1}(x)I'_{G_2}(x)\\
                         &=&I'_{G_1}(x)I'_{G_2}(x) \\
                         &=&I'_{G}(x).
\end{array}
\]
Therefore when $G-v$ is disconnected, we get $I'_G(x)=I'_{G-v}(x)+x I'_{G/v}(x)$.

The first formula$\colon$By the above, for any negative edge set $\mathcal{S}\subset \mathcal{E}_{-}(G)$, we have $I'_{G\setminus \mathcal{S}}=I'_{(G\setminus \mathcal{S})-v}+x I'_{(G\setminus \mathcal{S})/v}$. Therefore,
\begin{eqnarray*}
I^{+}_{G}(x)
&=&\sum_{\mathcal{S} \subseteq \mathcal{E}_{-}(G)}(-1)^{|\mathcal{S}|}I'_{G \setminus \mathcal{S}}(x) \\
&=&\sum_{\mathcal{S} \subseteq \mathcal{E}_{-}(G)}(-1)^{|\mathcal{S}|}\left(I'_{(G\setminus \mathcal{S})-v}(x)+x I'_{(G\setminus \mathcal{S})/v}(x)\right) \\
&=&\sum_{\mathcal{S} \subseteq \mathcal{E}_{-}(G)}(-1)^{|\mathcal{S}|}\left(I'_{(G-v)\setminus \mathcal{S}}(x)+x I'_{(G/v)\setminus \mathcal{S}}(x)\right) \\
&=&I^+_{G-v}(x)+x I^+_{G/v}(x).
\end{eqnarray*}

The second formula$\colon$From Lemma \ref{lem:sumint}, we have 
\begin{eqnarray*}
I^{+}_{G}(x)
&=&I^{+}_{G+\epsilon_1}(x)-I^{+}_{G\setminus \epsilon_1}(x)\\
&=&I^{+}_{G+\epsilon_1+\epsilon_2}(x)-I^{+}_{G+\epsilon_1\setminus \epsilon_2}(x)-I^{+}_{G+\epsilon_2\setminus \epsilon_1}(x)+I^{+}_{G\setminus \epsilon_1\setminus \epsilon_2}(x).
\end{eqnarray*}
By the first formula, we have $I^{+}_{G+\epsilon_1+\epsilon_2}(x)=I^{+}_{G-v}(x)+x I^+_{G/v}(x)$. From Proposition \ref{thm:nugatory}, we have $I^{+}_{G+\epsilon_1\setminus \epsilon_2}(x)=I^{+}_{G-v}(x)$ and $I^{+}_{G+\epsilon_2\setminus \epsilon_1}(x)=I^{+}_{G-v}(x)$. Since the signed bipartite graph $G\setminus \epsilon_1\setminus \epsilon_2$ is $G-v$ with an isolated vertex, we have $I^{+}_{G\setminus \epsilon_1\setminus \epsilon_2}(x)=(1-x)I^{+}_{G-v}(x)$. Therefore we have
\begin{eqnarray*}
I^{+}_{G}(x)
&=&I^{+}_{G-v}(x)+x I^+_{G/v}(x) -I^{+}_{G-v}(x)-I^{+}_{G-v}(x)+(1-x)I^{+}_{G-v}(x) \\
&=&-x I^+_{G-v}(x)+x I^+_{G/v}(x).
\end{eqnarray*}

The third formula$\colon$A similar argument yields. We have
\begin{eqnarray*}
I^{+}_{G}(x)
&=&I^{+}_{G+\epsilon_2}(x)-I^{+}_{G\setminus \epsilon_2}(x) \\
&=&I^{+}_{G-v}(x)+x I^+_{G/v}(x) -I^{+}_{G-v}(x) \\
&=&x I^+_{G/v}(x).
\end{eqnarray*}
This completes the proof.
\end{proof}

\section{The Ehrhart polynomial}\label{sec:ehr}
\subsection{Preliminaries}
In \cite{KP}, the interior polynomial of an unsigned bipartite graph $G$ is shown to be equivalent to the Ehrhart polynomial of the root polytope of $G$. We review some details. The root polytope of the bipartite graph is defined as follows.
\begin{definition}
Let $G=(V,E,\mathcal{E})$ be a bipartite graph. For $e \in E$ and $v \in V$, let {\bf e} and {\bf v} denote the corresponding standard generators of $\mathbb{R}^E \oplus \mathbb{R}^V$. Define the root polytope of $G$ by
\[
Q_G=\conv \{\, {\bf e} + {\bf v} \,|\, ev \mbox{ is an edge of }G \,\},
\]
where $\conv$ denotes the convex hull.
\end{definition}
See Figure \ref{k23} and \ref{rp} for an example.
\begin{figure}[htbp]
\begin{center}
\begin{tabular}{cc}\hspace{-15pt}
\begin{minipage}{0.5\hsize}
\begin{center}
\includegraphics[width=3.5cm]{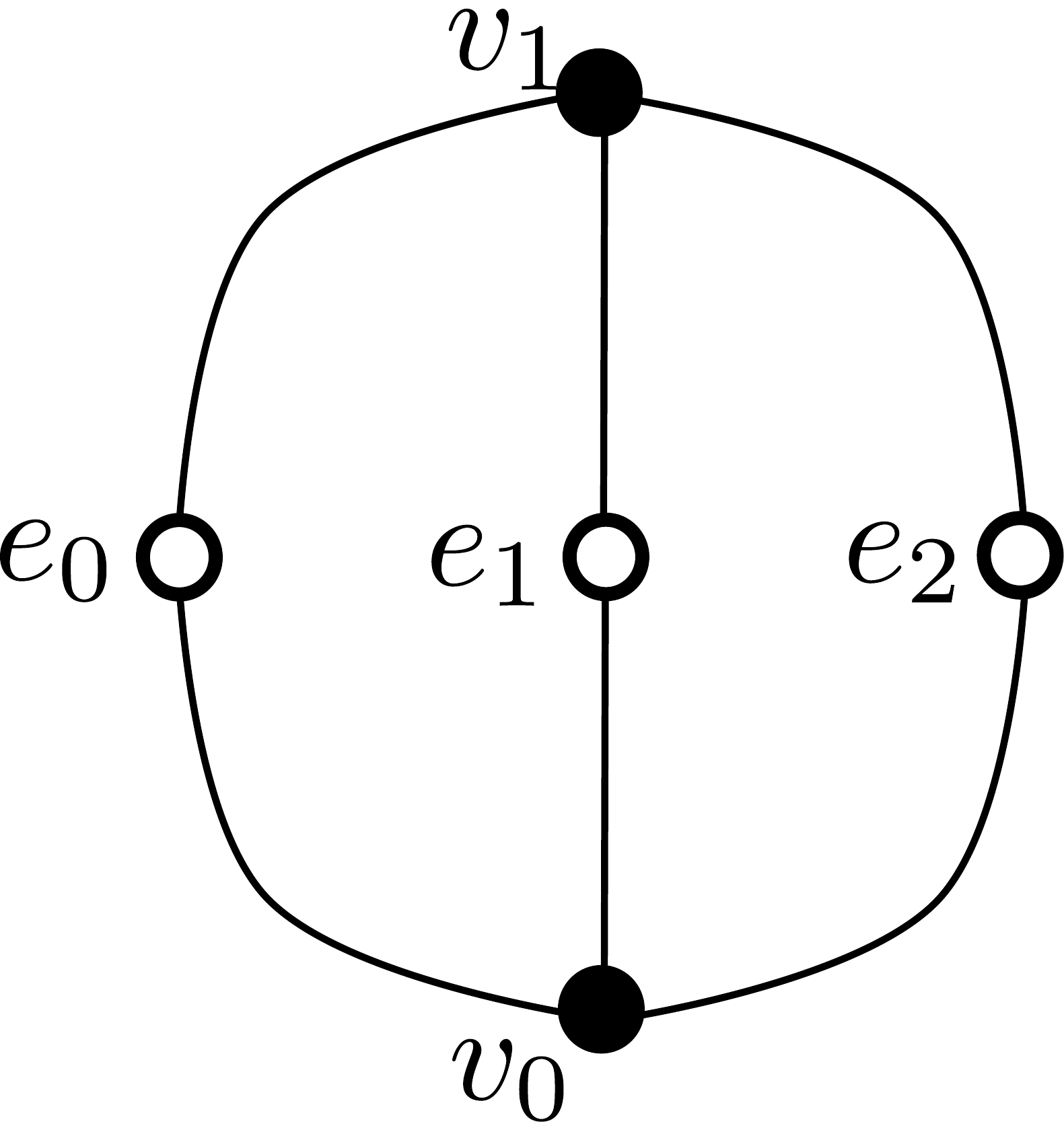}
\caption{The bipartite graph $K_{23}$.}\label{k23}
\end{center}
\end{minipage}
\begin{minipage}{0.55\hsize}
\begin{center}
\includegraphics[width=3.5cm]{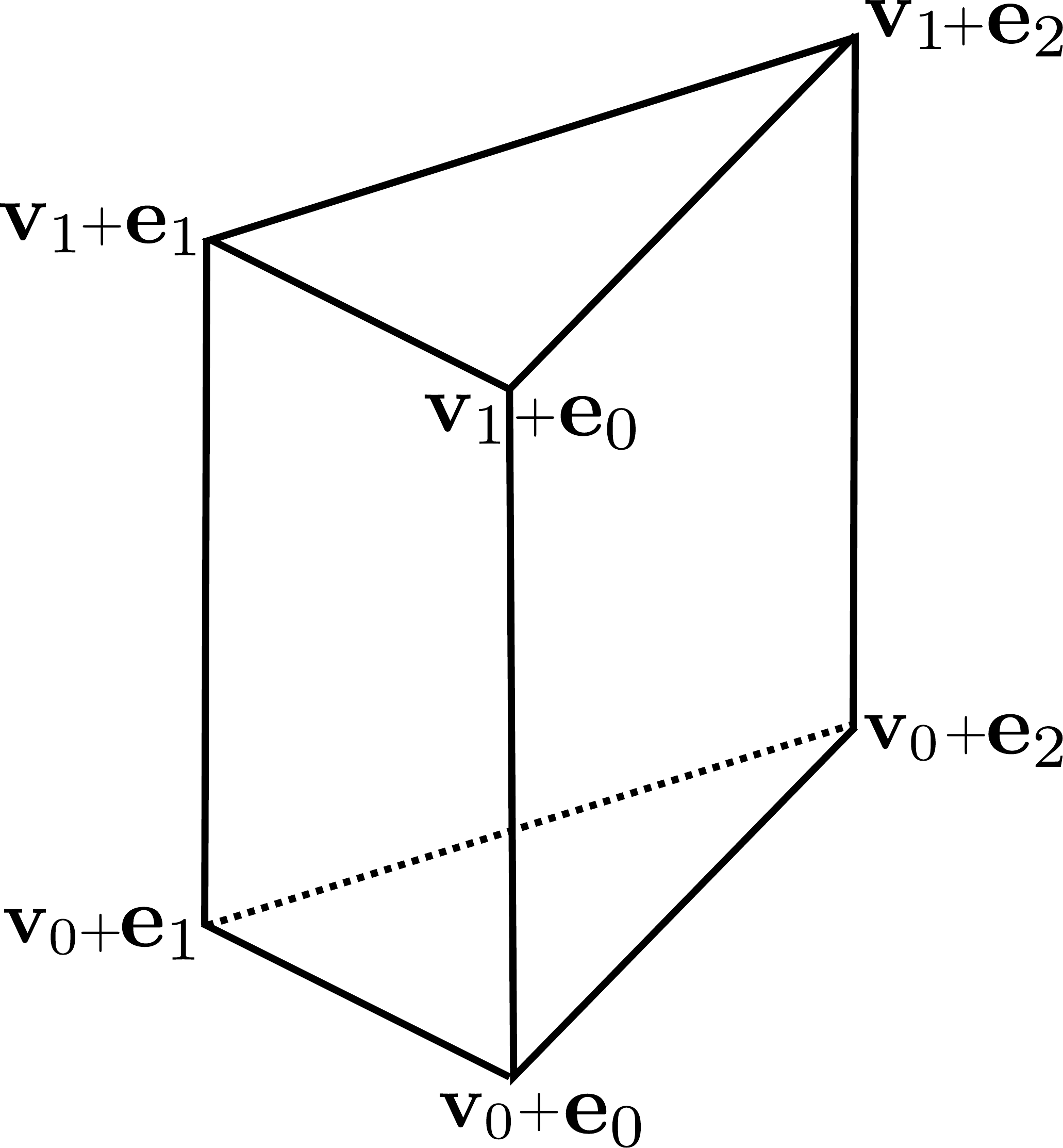}
\caption{The root polytope of $K_{23}$.}\label{rp}
\end{center}
\end{minipage}
\end{tabular}
\end{center}
\end{figure}

We will denote by $\bm{\epsilon}\in \mathbb{R}^E \oplus \mathbb{R}^V$ the vertex in $Q_G$ corresponding to the edge $\epsilon \in \mathcal{E}$. Thus, for any ${\bf x} \in Q_G$, we get $\displaystyle {\bf x}=\sum_{\epsilon \in \mathcal{E}}\lambda_{\epsilon}\bm{\epsilon}$, where each $\lambda_{\epsilon}\ge 0$ and $\displaystyle \sum_{\epsilon \in \mathcal{E}}\lambda_{\epsilon}=1$. We regard $\lambda_\epsilon$ as the weight of the edge $\epsilon$ in $G$, but note that the weights for a given ${\bf x} \in Q_G$ are not unique because $\{ \bm{\epsilon}\}_{\epsilon \in \mathcal{E}}$ is not basis.
\begin{definition}\label{def:weight}
If the map $\mathcal{W} \colon \mathcal{E}\to \mathbb{R}$ satisfies the following conditions, we call $\mathcal{W}$ a weight system.
\begin{enumerate}
\item For any $\epsilon \in \mathcal{E}$, we have $\mathcal{W}(\epsilon) \ge 0$. \\
\item $\displaystyle \sum_{\epsilon \in \mathcal{E}}\mathcal{W}(\epsilon)=1$.
\end{enumerate}
\end{definition}
Then any weight system $\mathcal{W}$ represents $\displaystyle {\bf x}=\sum_{\epsilon \in \mathcal{E}}\mathcal{W}(\epsilon)\bm{\epsilon} \in Q_G$. Two weight systems of $G$ represent the same point of $Q_G$ if and only if the sums of the weights around each  vertex of $G$ are the same.
\begin{example}\label{ex:cycchange}
Let $G$ be a bipartite graph which contains a cycle.  In the weight system $\mathcal{W}$, the weights along the cycle are non-negative real numbers $\mu_1,\lambda_1,\mu_2,\lambda_2,\cdots,\mu_n,\lambda_n$ as in Figure \ref{fig:W}. Let $\lambda=\min\{ \lambda_1,\cdots,\lambda_n\}$. The map $\mathcal{W}'$ is obtained from $\mathcal{W}$ by adding and subtracting  $\lambda$ alternately as in Figure \ref{fig:W2}. Because $\mathcal{W}'$ satisfies the conditions in Definition \ref{def:weight}, the map $\mathcal{W}'$ is a weight system of $G$. Notice that the sums of the weights of $\mathcal{W}'$ around each vertex of $G$ are the same as for $\mathcal{W}$. Thus, the two weight systems represent the same point in $Q_G$.
\begin{figure}[htbp]
\begin{center}
\begin{tabular}{cc}
\begin{minipage}{0.49\hsize}
\begin{center}
\includegraphics[height=3.5cm]{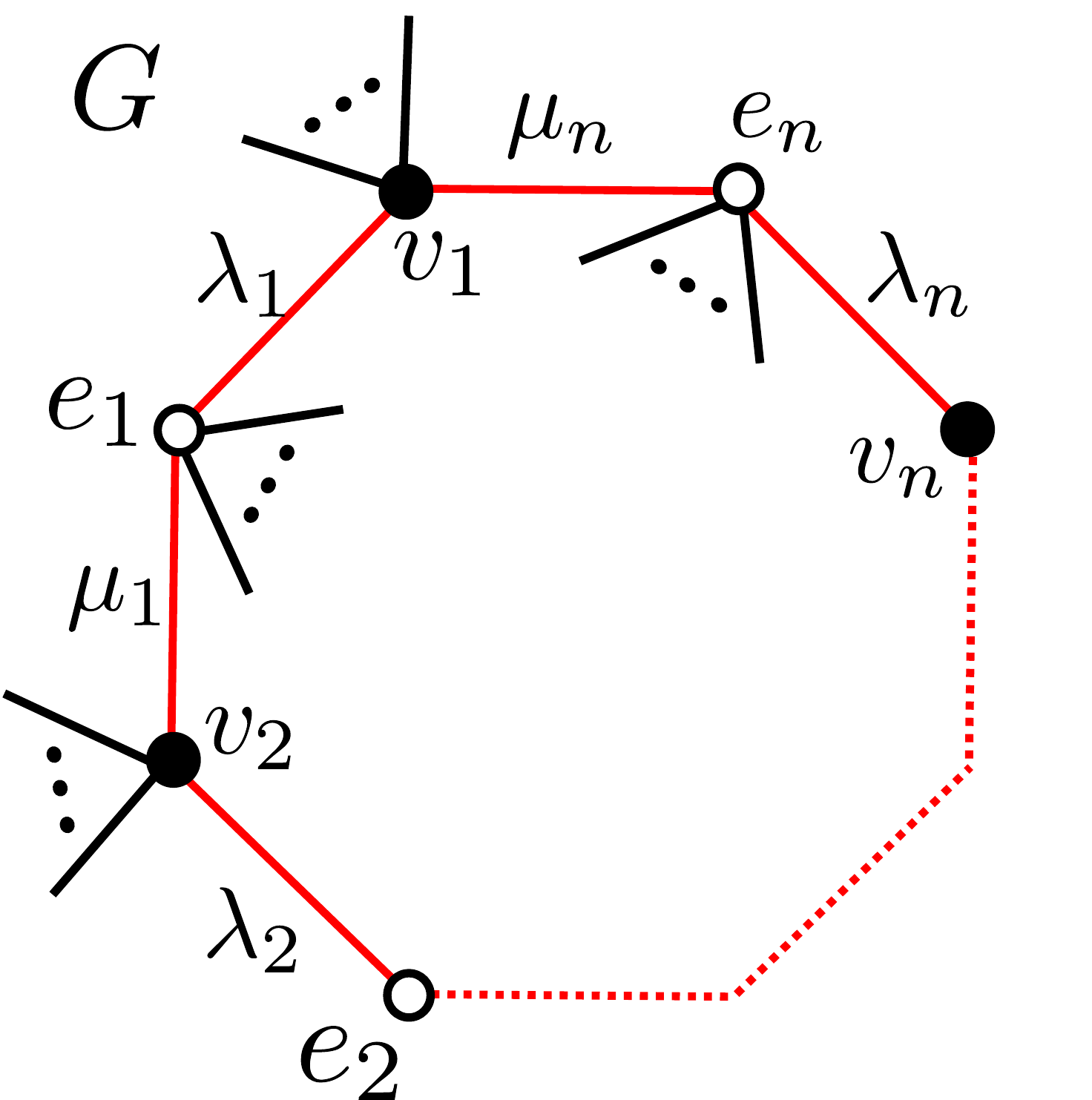}
\caption{The weight system $\mathcal{W}$.}
\label{fig:W}
\end{center}
\end{minipage}
\begin{minipage}{0.51\hsize}
\begin{center}
\includegraphics[height=3.5cm]{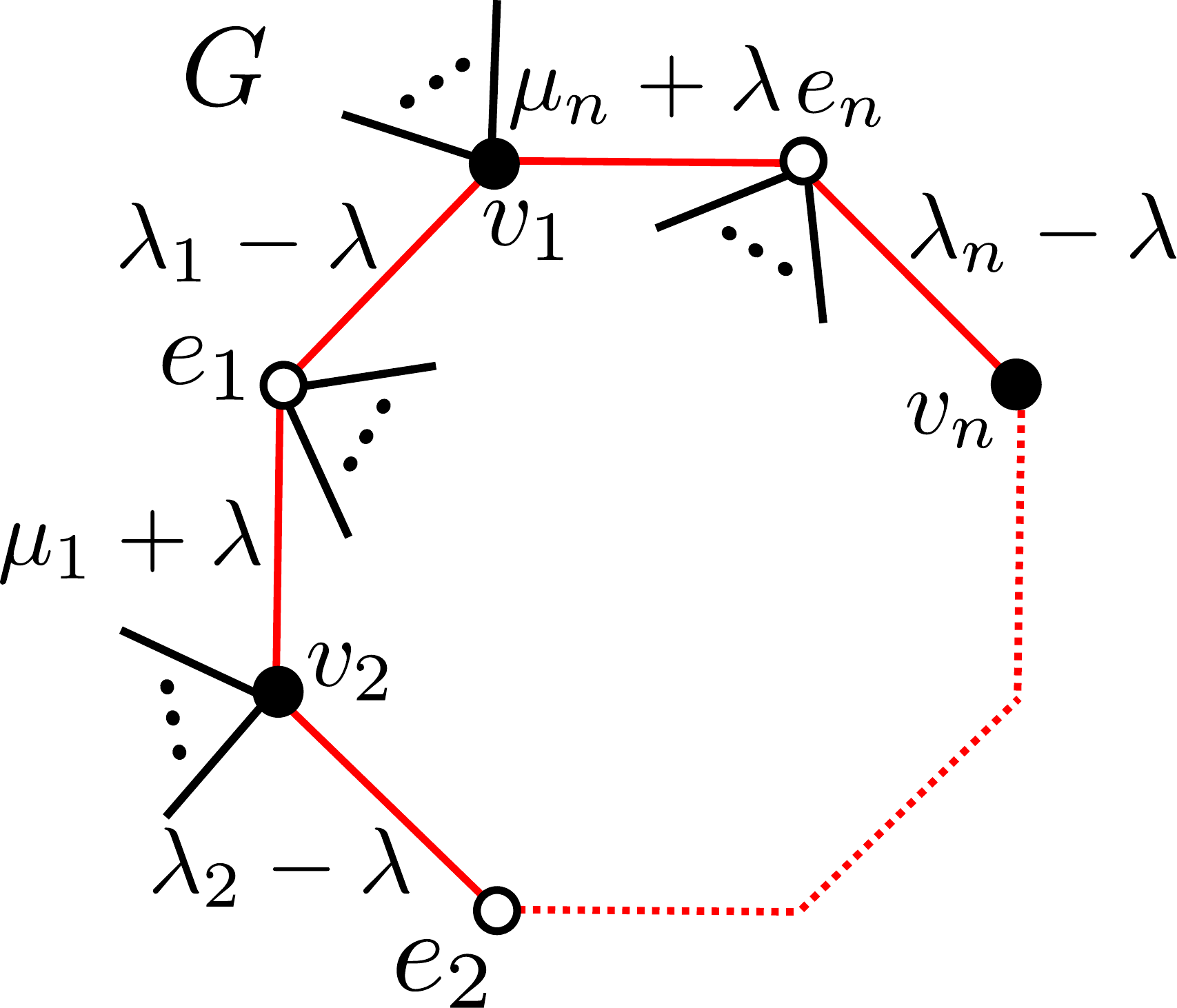}
\caption{The weight system $\mathcal{W}'$.}
\label{fig:W2}
\end{center}
\end{minipage}
\end{tabular}
\end{center}
\end{figure}
\end{example}
We call the change from $\mathcal{W}$ to $\mathcal{W}'$ described in Example \ref{ex:cycchange} a cycle change.
\begin{definition}
Let $G=(V,E,\mathcal{E})$ be a bipartite graph and $Q_G$ be the root polytope of $G$. For any non-negative integer $s$, the Ehrhart polynomial is defined by
\[
\varepsilon_{Q_G}(s) = |(s\cdot Q_G) \cap (\mathbb{Z}^{E}\oplus\mathbb{Z}^{V})|.
\]
\end{definition}
In general, for any polytope $P$, the analogously defined $\varepsilon_{P}(s)$ is not polynomial. However, for a convex polytope $P$ whose vertices are integer points, $\varepsilon_{P}(s)$ is a polynomial. Thus, $\varepsilon_{Q_G}(s)$ is a polynomial.
Moreover, for any lattice polytope, the degree of the Ehrhart polynomial is the dimension of the polytope. Since the dimension of the root polytope $Q_G$ of a bipatite graph with at least one edge is less than or equal to $|E|+|V|-2$, the degree of the Ehrhart polynomial $\varepsilon_{Q_G}(s)$ is less than or equal to $|E|+|V|-2$. The binomial coefficients $\left\{ \binom{s+|E|+|V|-2-k}{|E|+|V|-2} \right\}_{k=0}^{|E|+|V|-2}$ form a basis of the space {\boldmath$P$}$_{|E|+|V|-2}$ of polynomials with rational coefficients and of degree no greater than $|E|+|V|-2$ (see \cite[Lemma 3.8]{KP}). So we have unique rational numbers $\{a_k\}_{k=0}^{|E|+|V|-2}$ such that
\[
\varepsilon_{Q_G}(s)=\sum_{k=0}^{|E|+|V|-2}a_k\binom{s+|E|+|V|-2-k}{|E|+|V|-2}.
\]
Note that, for the single-vertex graph without edges, the root polytope is the empty set and the Ehrhart polynomial is zero. With this setup, K\'alm\'an and Postnikov prove the following theorem.
\begin{theorem}[\cite{KP}]\label{thm:ehrint}
Let $\varepsilon_{Q_G}(s) = \displaystyle\sum_{k=0}^{|E|+|V|-2}a_k\binom{s+|E|+|V|-2-k}{|E|+|V|-2}$ be the Ehrhart polynomial of the root polytope of a connected bipartite graph $G=(V,E,\mathcal{E})$ with at least one edge. Then either hypergraph induced by $G$ has the interior polynomial
\[
I_G(x)=a_0x^0+a_1x^1+\cdots+a_{|E|+|V|-2}x^{|E|+|V|-2}.
\]
\end{theorem}
In particular, the coefficients $a_k$ are non-negative integers, with several zeros at the end (after $k=\min\{|E|-1,|V|-1\}$, cf.\ \cite[Proposition 6.1]{K}).

We introduce a new description of the equivalence of the interior polynomial and the Ehrhart polynomial. Moreover, the following theorem is true for a single-vertex graph without edges. The Ehrhart series is the power series obtained from the Ehrhart polynomial as follows. 
\begin{definition}
Let $G$ be a bipartite graph and $\varepsilon_{Q_G}(s)$ be the Ehrhart polynomial of the root polytope $Q_G$. The Ehrhart series is defined by
\[
\ehr_{Q_G}(x)=1+\sum_{s\in \mathbb{N}}\varepsilon_{Q_G}(s)x^s.
\]
\end{definition}
Notice that for a (bipartite) graph with no edges, we have $\ehr(x)=1$.
\begin{theorem}\label{lem:serint}
Let $G=(V,E,\mathcal{E})$ be a connected bipartite graph and $I_G(x)$ be the interior polynomial of $G$. Then
\[
\frac{I_G(x)}{(1-x)^{|E|+|V|-1}}=\ehr_{Q_G}(x).
\]
\end{theorem}
\begin{proof}
If a bipartite graph $G$ is a single-vertex graph without edges, we have $\ehr_{Q_G}(x)=1$. Thus, we have $(1-x)^{1-1}\ehr_{Q_G}(x)=1$. This is the interior polynomial $I_G(x)$. Therefore the theorem is true for a single-vertex graph without edges.

Assume that a bipartite graph has edges. Write $d=|E|+|V|-2$. It follows easily by using the geometric series that
\[
\frac{1}{(1-x)^{d+1}}=\sum_{s\in\mathbb{Z}_{\ge0}}\binom{s+d}{d} x^s .
\]
Furthermore using Theorem \ref{thm:ehrint},
\begin{eqnarray*}
\frac{I_G(x)}{(1-x)^{d+1}}&=&\left(1+ \sum_{s\in\mathbb{N}}\binom{s+d}{d} x^s \right) (a_0x^0+a_1x^1+\cdots+a_{d}x^{d})\\
&=&a_0+\sum_{s\in\mathbb{N}}\left(a_0\binom{s-0+d}{d}+a_1\binom{s-1+d}{d}+\cdots+a_d\binom{s-d+d}{d}\right)x^s\\
&=&a_0+\sum_{s\in\mathbb{N}}\left(\sum_{k=0}^{d} a_k\binom{s-k+d}{d}\right)x^s\\
&=&a_0+\sum_{s\in\mathbb{N}} \varepsilon_{Q_G}(s) x^s.
\end{eqnarray*}
By \cite[Proposition 6.2]{K}, the constant term of the interior polynomial is $1$. Therefore we have 
\[
\frac{I_G(x)}{(1-x)^{d+1}}=1+\sum_{s\in\mathbb{N}} \varepsilon_{Q_G}(s) x^s=\ehr_{Q_G}(x).
\]
This completes the proof.
\end{proof}
Thanks to this description, we do not have to change bases.
\subsection{A disconnected version}
Before we discuss the Ehrhart polynomial of signed bipartite graphs, we discuss the Ehrhart polynomial of disconnected bipartite graphs. This is not as easy as the analogous discussion of the interior polynomial. First, we introduce a claim about the Ehrhart series of a disjoint union. 
\begin{lemma}\label{thm:serdecompos}
Let $G=G_1\cup G_2$ be the disjoint union of the connected bipartite graphs $G_1$ and $G_2$. Then
\[
\ehr_{Q_G}(x)=\ehr_{Q_{G_1}}(x)\ehr_{Q_{G_2}}(x).
\]
\end{lemma}
\begin{proof}
If the bipartite graphs $G_1,G_2$ are both single-vertex graphs without edges, we have $\ehr_{Q_{G_1}}(x)=1$, and $\ehr_{Q_{G_2}}(x)=1$. Moreover, since the bipartite graph $G$, which is disjoint union of $G_1$ and $G_2$, is two isolated vertex, we have $\ehr_{Q_G}=1$. Therefore the theorem is true.

Assume either of the bipartite graphs $G_1,G_2$ has at least one edge. Let $G_1=(V_1,E_1,\mathcal{E}_1)$, and $G_2=(V_2,E_2,\mathcal{E}_2)$. We use Theorem \ref{lem:serint} for the connected bipartite graphs $G_1$ and $G_2$ to obtain
\[
\frac{I_{G_1}(x)}{(1-x)^{|E_1|+|V_1|-1}}=\ehr_{Q_{G_1}}(x), 
\hspace{5pt}\mbox{ and }\hspace{5pt}
\frac{I_{G_2}(x)}{(1-x)^{|E_2|+|V_2|-1}}=\ehr_{Q_{G_2}}(x).
\]
Let the bipartite graph $G'$ be defined by $(V_1\cup V_2,E_1\cup E_2,\mathcal{E}_1 \cup \mathcal{E}_2 \cup \{ \epsilon\})$, where $\epsilon$ is a new edge connecting two vertices which are chosen opportunely from each bipartite graph $G_1$ and $G_2$. Then $G'$ is connected and we may use the product formula (\cite[Theorem 6.7]{K}) $I_{G'}(x)=I_{G_1}(x)I_{G_2}(x)$. By the above,
\begin{eqnarray*}
\ehr_{Q_{G_1}}(x)\ehr_{Q_{G_2}}(x)
&=&\frac{I_{G_1}(x)}{(1-x)^{|E_1|+|V_1|-1}}\frac{I_{G_2}(x)}{(1-x)^{|E_2|+|V_2|-1}}\\
&=&(1-x)\frac{I_{G'}(x)}{(1-x)^{|E_1|+|E_2|+|V_1|+|V_2|-1}}.\\
\end{eqnarray*}
By Theorem \ref{lem:serint}, we have
\[
\frac{I_{G'}(x)}{(1-x)^{|E_1|+|E_2|+|V_1|+|V_2|-1}}=\ehr_{Q_{G'}}(x).
\]
So we get $\ehr_{Q_{G_1}}(x)\ehr_{Q_{G_2}}(x)=(1-x)\ehr_{Q_{G'}}(x)$.

Now we consider the root polytope $Q_{G'}$ and the Ehrhart polynomial $\varepsilon_{Q_{G'}}(s)$. First, we show that for any $s\in\mathbb{Z}_{\ge 0}$, we have $\displaystyle s\cdot Q_{G'}=\bigcup_{k\in[0,s]}\left(k\cdot Q_{\epsilon}\oplus(s-k)\cdot Q_G\right)$, where $Q_{\epsilon}$ is the singleton given as the root polytope of the bipartite graph consisting of only one edge $\epsilon$. As this is obvious for $s=0$, we may assume that $s\ge1$. We take $s{\bf x} \in s\cdot Q_{G'}$. We will denote by $n=|\mathcal{E}_1|+|\mathcal{E}_2|$ the number of edges in $G=G_1 \cup G_2$. Now, we denote by $\bm{\epsilon}_1,\bm{\epsilon}_2,\ldots ,\bm{\epsilon}_n$ the vectors corresponding to the edges of the bipartite graph $G=G_1\cup G_2$. Since the edges  $\epsilon,\epsilon_1,\ldots\epsilon_n$ are all the edges in the bipartite graph $G'$, we have
\[
{\bf x}=\lambda \bm{\epsilon} +\lambda_1\bm{\epsilon}_1+\lambda_2\bm{\epsilon}_2+\cdots +\lambda_n\bm{\epsilon}_n,
\]
where $\displaystyle \lambda+ \sum_{i=1}^n \lambda_i=1$ and $\lambda,\lambda_1,\ldots\lambda_n$ are non-negative. For ${\bf x}\ne\bm{\epsilon}$, we also have, 
\[
s{\bf x}=s\lambda \cdot \bm{\epsilon} +(s-s\lambda)\frac{\lambda_1\bm{\epsilon}_1+\lambda_2\bm{\epsilon}_2+\cdots +\lambda_n\bm{\epsilon}_n}{1-\lambda}.
\]
Since $\displaystyle\frac{\lambda_1+\lambda_2+\cdots+\lambda_{n}}{1-\lambda}=1$, we obtain $\displaystyle\frac{\lambda_1\bm{\epsilon}_1+\lambda_2\bm{\epsilon}_2+\cdots +\lambda_n\bm{\epsilon}_n}{1-\lambda} \in Q_G$. Letting $k=\lambda s$,
\[
s{\bf x}=k \bm{\epsilon}+(s-k) \frac{\lambda_1\bm{\epsilon}_1+\lambda_2\bm{\epsilon}_2+\cdots +\lambda_n\bm{\epsilon}_n}{1-\lambda}\in k\cdot Q_{\epsilon}\oplus(s-k)\cdot Q_{G} .
\]
If $0\le\lambda\le1$, then $0\le k \le s$. Thus $s\cdot Q_{G'}\subset \displaystyle\bigcup_{k\in[0,s]}\left(k\cdot Q_{\epsilon}\oplus(s-k)\cdot Q_{G}\right)$. Next we take $k\bm{\epsilon}+(s-k){\bf x} \in (kQ_{\epsilon}\oplus(s-k)Q_{G})$, where ${\bf x}\in Q_G$. Letting ${\bf x}=\lambda_1\bm{\epsilon}_1+\lambda_2\bm{\epsilon}_2+\cdots +\lambda_n\bm{\epsilon}_n$, we have
\[
k\bm{\epsilon}+(s-k){\bf x} =k\bm{\epsilon}+(s-k)(\lambda_1\bm{\epsilon}_1+\lambda_2\bm{\epsilon}_2+\cdots +\lambda_n\bm{\epsilon}_n),
\]
where $\lambda_i \ge0$ for all $i$ and $\displaystyle\sum_{i=1}^n\lambda_i=1$. We have
\[
k\bm{\epsilon}+(s-k){\bf x} =s\frac{k}{s}\bm{\epsilon}+s\frac{s-k}{s}(\lambda_1\bm{\epsilon}_1+\lambda_2\bm{\epsilon}_2+\cdots +\lambda_n\bm{\epsilon}_n).
\]
Since $\displaystyle\frac{k}{s}+\frac{s-k}{s}\sum_{k=1}^n\lambda_k=1$, we have 
\[
k\bm{\epsilon}+(s-k){\bf x} =s\left(\frac{k}{s}\bm{\epsilon}
+\frac{s-k}{s} (\lambda_1\bm{\epsilon}_1+\lambda_2\bm{\epsilon}_2+\cdots +\lambda_n\bm{\epsilon}_n)\right)\in s\cdot Q_{G'}.
\]
Hence we have shown that $\displaystyle\bigcup_{k\in[0,s]}\left(k\cdot Q_{\epsilon}\oplus(s-k)\cdot Q_{G}\right) \subset sQ_{G'}$. Therefore, $\displaystyle s\cdot Q_{G'}=\bigcup_{k\in[0,s]}\left(k\cdot Q_{\epsilon}\oplus(s-k)\cdot Q_{G}\right)$. The geometric idea behind this formula is that $Q_{G'}$ is a cone over $Q_G$ which we slice into levels parallel to the base.

Second, we show that, for any non-integer $k\in[0,s]$, we have 
\[
(k\cdot Q_{\epsilon}\oplus(s-k)\cdot Q_{G})\cap (\mathbb{Z}^{E_1}\oplus\mathbb{Z}^{E_2}\oplus\mathbb{Z}^{V_1}\oplus\mathbb{Z}^{V_2})=\emptyset.\]
Suppose that $k$ is non-integer. Any weight system $\mathcal{W}$ for an integer point in $s\cdot Q_{G'}=k\cdot Q_{\epsilon}\oplus(s-k)\cdot Q_{G}$ satisfies that $\mathcal{W}(\epsilon')\ge0$ for any $\epsilon' \in \mathcal{E}_1\cup\mathcal{E}_2 \cup\{ \epsilon \}$, that $\displaystyle \sum_{\epsilon'\in\mathcal{E}_1\cup\mathcal{E}_2 \cup\{ \epsilon \}}\mathcal{W}(\epsilon')=s$, and that the sum of the weights around each vertex is integer. Moreover non-integer $k$ is the weight of the edge $\epsilon$. Then we consider the vertex $v \in G_1$ incident to $\epsilon$. Since the sum of weights around $v$ is integer, we may take an edge of $G_1$ which is incident to $v$ and whose weight is not integer. Repeating this operation, we find a path in $G_1$ of edges with non-integer weights. Since $G_1$ is finite, we get a cycle. Using cycle change along this cycle, the weight of at least one edge can be changed to an integer (such as $0$). Again $G_1$ is finite, this operation can not be repeated indefinitely and hence there can be no integer points in $s\cdot Q_{G'}$ at a non-integer level $k$. Thus $\displaystyle |(k\cdot Q_{\epsilon}\oplus(s-k)\cdot Q_{G})\cap( \mathbb{Z}^{E_1}\oplus\mathbb{Z}^{E_2}\oplus\mathbb{Z}^{V_1}\oplus\mathbb{Z}^{V_2})|\ne0$ if and only if $k$ is integer. Since $Q_{\epsilon}$ is a one point set, we have
\begin{eqnarray*}
\varepsilon_{Q_{G'}}(s)
&=&|(s\cdot Q_{G'})\cap(
\mathbb{Z}^{E_1}\oplus\mathbb{Z}^{E_2}\oplus\mathbb{Z}^{V_1}\oplus\mathbb{Z}^{V_2})|\\
&=&\sum_{k=0}^s|(k\cdot Q_{\epsilon}\oplus (s-k)\cdot Q_{G})\cap(
\mathbb{Z}^{E_1}\oplus\mathbb{Z}^{E_2}\oplus\mathbb{Z}^{V_1}\oplus\mathbb{Z}^{V_2})|\\
&=&\varepsilon_{Q_G}(s)+\varepsilon_{Q_G}(s-1)+\cdots+\varepsilon_{Q_G}(0).
\end{eqnarray*}
Lastly, by the above,
\begin{eqnarray*}
\ehr_{Q_{G_1}}(x)\ehr_{Q_{G_2}}(x)
&=&(1-x)\ehr_{Q_{G'}}(x)\\
&=&(1-x)\left(1+\sum_{s\in\mathbb{N}}\varepsilon_{Q_{G'}}(s)x^s\right)\\
&=&(1-x)\left(1+\sum_{s\in\mathbb{N}}(\varepsilon_{Q_G}(s)+\varepsilon_{Q_G}(s-1)+\cdots+\varepsilon_{Q_G}(0)) x^s\right).
\end{eqnarray*}
By definition, for any bipartite graph $G$, we have $\varepsilon_{Q_G}(0)=1$. Therefore we have
\begin{eqnarray*}
\ehr_{Q_{G_1}}(x)\ehr_{Q_{G_2}}(x)
&=&(1-x)\left(1+\sum_{s\in\mathbb{N}}(\varepsilon_{Q_G}(s)+\varepsilon_{Q_G}(s-1)+\cdots+\varepsilon_{Q_G}(1)+1) x^s\right)\\
&=&(1-x)\left(1+\sum_{s\in\mathbb{N}}\varepsilon_{Q_G}(s)x^s\right)(1+x^1+x^2+\cdots)\\
&=&(1-x)\left(1+\sum_{s\in\mathbb{N}}\varepsilon_{Q_G}(s)x^s\right)\frac{1}{1-x}\\
&=&1+\sum_{s\in\mathbb{N}}\varepsilon_{Q_G}(s)x^s\\
&=&\ehr_{Q_G}(x).
\end{eqnarray*}
This completes the proof.
\end{proof}
Furthermore we prove the following generalizations of Lemma \ref{thm:serdecompos} and Theorem \ref{lem:serint} for any (possibly disconnected but unsigned) bipartite graph.
\begin{theorem}\label{thm:serdec}
Let $G=G_1\cup G_2$ be the disjoint union of (possibly disconnected) bipartite graphs $G_1$ and $G_2$. Then
\[
\ehr_{Q_G}(x)=\ehr_{Q_{G_1}}(x)\ehr_{Q_{G_2}}(x).
\]
\end{theorem}
\begin{theorem}\label{thm:disserint}
Let $G=(V,E,\mathcal{E})$ be a (possibly disconnected) bipartite graph and $I'_G(x)$ be the interior polynomial of $G$. Then
\[
\frac{I'_G(x)}{(1-x)^{|E|+|V|-1}}=\ehr_{Q_G}(x).
\]
\end{theorem}

\begin{proof}[Proof of Theorem \ref{thm:serdec} and Theorem \ref{thm:disserint}]
We proceed by interlocking induction on the component number $k(G)$. When $k(G)=1$, the statement of Theorem \ref{thm:disserint} holds by Theorem \ref{lem:serint}. When $k(G)=2$, the statement of Theorem \ref{thm:serdec} holds by Lemma \ref{thm:serdecompos}.

Suppose that the theorems holds when $k(G)<n$ and let $G$ have $k(G)=n$ components. We take the non-empty bipartite graphs $G_1,G_2$ such that $G$ is their disjoint union.

First, we prove the Theorem \ref{thm:serdec} for $G=G_1\cup G_2$. If the bipartite graphs $G_1,G_2$ are both vertices with no edges, we have $\ehr_{Q_{G_1}}(x)=1$, and $\ehr_{Q_{G_2}}(x)=1$. And by $\ehr_{Q_G}(s)=1$, the theorem is true.

Assume either of the bipartite graphs $G_1,G_2$ has at least one edge. Let $G_1=(V_1,E_1,\mathcal{E}_1)$, and $G_2=(V_2,E_2,\mathcal{E}_2)$. By induction hypothesis, for the bipartite graph which has at most $n-1$ components, the statement of Theorem \ref{thm:disserint} is true. Since $G_1$ and $G_2$ have at most $n-1$ components, we have
\[
\frac{I'_{G_1}(x)}{(1-x)^{|E_1|+|V_1|-1}}=\ehr_{Q_{G_1}}(x), 
\hspace{5pt}\mbox{ and }\hspace{5pt}
\frac{I'_{G_2}(x)}{(1-x)^{|E_2|+|V_2|-1}}=\ehr_{Q_{G_2}}(x).
\]
Define the bipartite graph $G'=(V_1\cup V_2,E_1\cup E_2,\mathcal{E}_1 \cup \mathcal{E}_2 \cup \{ \epsilon\})$, where $\epsilon$ is a new edge connecting two vertices which are chosen opportunely from each bipartite graph $G_1$ and $G_2$. Then $G'$ has $n-1$ components and we may use the product formula $I'_{G'}(x)=I'_{G_1}(x)I'_{G_2}(x)$ (an unsigned case of Theorem \ref{thm:product}). By the above,
\begin{eqnarray*}
\ehr_{Q_{G_1}}(x)\ehr_{Q_{G_2}}(x)
&=&\frac{I_{G_1}(x)}{(1-x)^{|E_1|+|V_1|-1}}\frac{I_{G_2}(x)}{(1-x)^{|E_2|+|V_2|-1}}\\
&=&(1-x)\frac{I_{G'}(x)}{(1-x)^{|E_1|+|E_2|+|V_1|+|V_2|-1}}.\\
\end{eqnarray*}
By the induction hypothesis on the side of Theorem \ref{thm:disserint}, we have
\[
\frac{I_{G'}(x)}{(1-x)^{|E_1|+|E_2|+|V_1|+|V_2|-1}}=\ehr_{Q_{G'}}(x).
\]
So we get $\ehr_{Q_{G_1}}(x)\ehr_{Q_{G_2}}(x)=(1-x)\ehr_{Q_{G'}}(x)$. Therefore it is sufficient to show that $(1-x)\ehr_{Q_{G'}}(x)=\ehr_{Q_G}(x)$, which is proved in same way as Lemma \ref{thm:serdecompos}.

Second, we prove the statement of Theorem \ref{thm:disserint} when $k(G)=n$. By the above, we have
\[
\ehr_{Q_G}(x)=\ehr_{Q_{G_1}}(x)\ehr_{Q_{G_2}}(x)
\]
By induction hypothesis on the side of Theorem \ref{thm:disserint}, we obtain
\[
\ehr_{Q_G}(x)=\frac{I'_{G_1}(x)}{(1-x)^{|E_1|+|V_1|-1}}\frac{I'_{G_2}(x)}{(1-x)^{|E_2|+|V_2|-1}}
\]
Furthermore by Lemma \ref{lem:disint},
\begin{eqnarray*}
\ehr_{Q_G}(x)=
&=&\frac{1-x}{1-x}\frac{I'_{G_1}(x)}{(1-x)^{|E_1|+|V_1|-1}}\frac{I'_{G_2}(x)}{(1-x)^{|E_2|+|V_2|-1}}\\
&=&\frac{I'_{G_1\cup G_2}(x)}{(1-x)^{|E_1|+|E_2|+|V_1|+|V_2|-1}}\\
&=&\frac{I'_G(x)}{(1-x)^{|E|+|V|-1}}.
\end{eqnarray*}
This completes the proof by induction.
\end{proof}

\begin{cor}\label{thm:disintehr2}
Let $\varepsilon_{Q_G}(s) = \displaystyle \sum_{k=0}^{|E|+|V|-2}a_k\binom{s+|E|+|V|-2-k}{|E|+|V|-2}$ be the Ehrhart polynomial of a (possibly disconnected) bipartite graph $G$ with at least one edge.
Then
\[
I'_G(x)=a_0x^0+a_1x^1+\cdots+a_{|E|+|V|-2}x^{|E|+|V|-2}.
\]
\end{cor}

\begin{proof}
By Theorem \ref{thm:disserint}, we have
\[
\frac{I'_G(x)}{(1-x)^{|E|+|V|-1}}=\ehr_{Q_G}(x).
\]
Assume that $G$ has $k(G)$ components. The dimension of the root polytope $Q_G$ is less than or equal to $|E|+|V|-2$. Then letting $\displaystyle\varepsilon_{Q_G}(s) = \sum_{k=0}^{|E|+|V|-2}a_k\binom{s+|E|+|V|-2-k}{|E|+|V|-2}$, the right hand side is
\[
\ehr_{Q_G}(x)=
\sum_{s\in\mathbb{Z}_{\ge0}}
\left( \sum_{k=0}^{|E|+|V|-2}a_k\binom{s+|E|+|V|-2-k}{|E|+|V|-2}\right)x^s.
\]
By the definition of $I'$, the degree of the interior polynomial $I'_G(x)$ is at most $|E|+|V|-k(G)-1$. This is at most $|E|+|V|-2$. Letting $I'_G(x)=b_0x^0+b_1x^1+\cdots+b_{|E|+|V|-2}x^{|E|+|V|-2}$, the left hand side of the above is
\begin{eqnarray*}
\frac{I'_G(x)}{(1-x)^{|E|+|V|-1}}
&=&\left( \sum_{s\in\mathbb{Z}_{\ge0}}\binom{s+|E|+|V|-2}{|E|+|V|-2} x^s \right) (b_0x^0+b_1x^1+\cdots+b_{|E|+|V|-2}x^{|E|+|V|-2})\\
&=&\sum_{s\in\mathbb{Z}_{\ge0}}\left(\sum_{k=0}^{|E|+|V|-2} b_k\binom{s+|E|+|V|-2-k}{|E|+|V|-2}\right)x^s.
\end{eqnarray*}
Since the binomial coefficients $\left\{ \binom{s+|E|+|V|-2-k}{|E|+|V|-2} \right\}_{k=0}^{|E|+|V|-2}$ form a basis of the space {\boldmath$P$}$_{|E|+|V|-2}$, we have $a_k=b_k$ for any $k=0,\cdots,|E|+|V|-2$. Hence we get the statement of the corollary.
\end{proof}

\subsection{A signed version}
The Ehrhart polynomial and the Ehrhart series of a signed bipartite graph is defined in the same way as the interior polynomial.
\begin{definition}
Let $G=(V,E,\mathcal{E})$ be a signed bipartite graph.
We define the signed Ehrhart polynomial as follows.
\[
\varepsilon^{+}_{G}\left(x \right)=\sum_{\mathcal{S} \subseteq \mathcal{E}_{-}(G)}(-1)^{|\mathcal{S}|}\varepsilon_{Q_{G \setminus \mathcal{S}}}\left(x \right).
\]
We define the signed Ehrhart series by the following.
\[
\ehr^+_{G}(x)=\sum_{\mathcal{S} \subseteq \mathcal{E}_{-}(G)}(-1)^{|\mathcal{S}|}\ehr_{Q_{G \setminus \mathcal{S}}}\left(x \right).
\]
In both formulae, the graph $G\setminus\mathcal{S}$ is treated as unsigned.
\end{definition}
We prove that the signed interior polynomial is equivalent to the signed Ehrhart polynomial using Corollary \ref{thm:disintehr2}.
\begin{theorem}\label{thm:signehrint}
Let $\displaystyle \varepsilon^+_{G}(s)=\sum_{k=0}^{|E|+|V|-2}a_k\binom{s+|E|+|V|-2-k}{|E|+|V|-2}$ be the signed Ehrhart polynomial of the signed bipartite graph $G=(V,E,\mathcal{E})$ with at least one positive and one negative edge.
Then
\[
I^+_G(x)=a_0x^0+a_1x^1+\cdots+a_{|E|+|V|-2}x^{|E|+|V|-2}.
\]
\end{theorem}
\begin{proof}
Let
\[
\varepsilon_{Q_{G \setminus \mathcal{S}}}\left(x \right)=\sum_{k=0}^{|E|+|V|-2}a^{G\setminus\mathcal{S}}_k\binom{s+|E|+|V|-2-k}{|E|+|V|-2}.
\]
By definition, we have
\begin{eqnarray*}
\varepsilon^+_{G}(s)
&=&\sum_{\mathcal{S} \subseteq \mathcal{E}_{-}(G)}(-1)^{|\mathcal{S}|}\sum_{k=0}^{|E|+|V|-2}a^{G\setminus\mathcal{S}}_k\binom{s+|E|+|V|-2-k}{|E|+|V|-2}\\
&=&\sum_{k=0}^{|E|+|V|-2}\left(\sum_{\mathcal{S} \subseteq \mathcal{E}_{-}(G)}(-1)^{|\mathcal{S}|}a^{G\setminus\mathcal{S}}_k\right)\binom{s+|E|+|V|-2-k}{|E|+|V|-2}.
\end{eqnarray*}
On the other hand, since the unsigned bipartite graph $G\setminus\mathcal{S}$ has at least one edge, by Corollary \ref{thm:disintehr2}, we have
\[
I'_{G\setminus\mathcal{S}}(x)=a^{G\setminus\mathcal{S}}_0x^0+a^{G\setminus\mathcal{S}}_1x^1+\cdots+a^{G\setminus\mathcal{S}}_{|E|+|V|-2}x^{|E|+|V|-2}. 
\]
By the definition of the signed interior polynomial, we have
\begin{eqnarray*}
I^+_G(x)
&=&\sum_{\mathcal{S} \subseteq \mathcal{E}_{-}(G)}(-1)^{|\mathcal{S}|}\sum_{k=0}^{|E|+|V|-2}a^{G\setminus\mathcal{S}}_kx^k \\
&=&\sum_{k=0}^{|E|+|V|-2}\sum_{\mathcal{S} \subseteq \mathcal{E}_{-}(G)}(-1)^{|\mathcal{S}|}a^{G\setminus\mathcal{S}}_kx^k .
\end{eqnarray*}
Therefore we obtain the statement of the theorem.
\end{proof}
We prove that the signed interior polynomial is equivalent to the signed Ehrhart series using Theorem \ref{thm:disserint}. Moreover we can apply following theorem for more general signed bipartite graphs. 
\begin{theorem}\label{thm:signserint}
Let $G=(V,E,\mathcal{E})$ be a signed bipartite graph and $I^+_G(x)$ be the signed interior polynomial of $G$. Then
\[
\frac{I^+_G(x)}{(1-x)^{|E|+|V|-1}}=\ehr^+_{G}(x).
\]
\end{theorem}
\begin{proof}
By definition and Theorem \ref{thm:disserint},
\begin{eqnarray*}
\ehr^+_{G}(x)
&=&\sum_{\mathcal{S} \subseteq \mathcal{E}_{-}(G)}(-1)^{|\mathcal{S}|}\ehr_{Q_{G \setminus \mathcal{S}}}\left(x \right)\\
&=&\sum_{\mathcal{S} \subseteq \mathcal{E}_{-}(G)}(-1)^{|\mathcal{S}|}\frac{I'_{G\setminus \mathcal{S}}(x)}{(1-x)^{|E|+|V|-1}}\\
&=&\frac{I^+_{G\setminus \mathcal{S}}(x)}{(1-x)^{|E|+|V|-1}}.
\end{eqnarray*}
This completes the proof.
\end{proof}

\subsection{An alternating cycle and a vanishing formula}\label{sec:alt}
An alternating cycle is a cycle in a bipartite graph containing positive edges and negative edges alternately like in Figure \ref{fig:altcycle}.
\begin{figure}[htbp]
\begin{center}
\includegraphics[width=3cm]{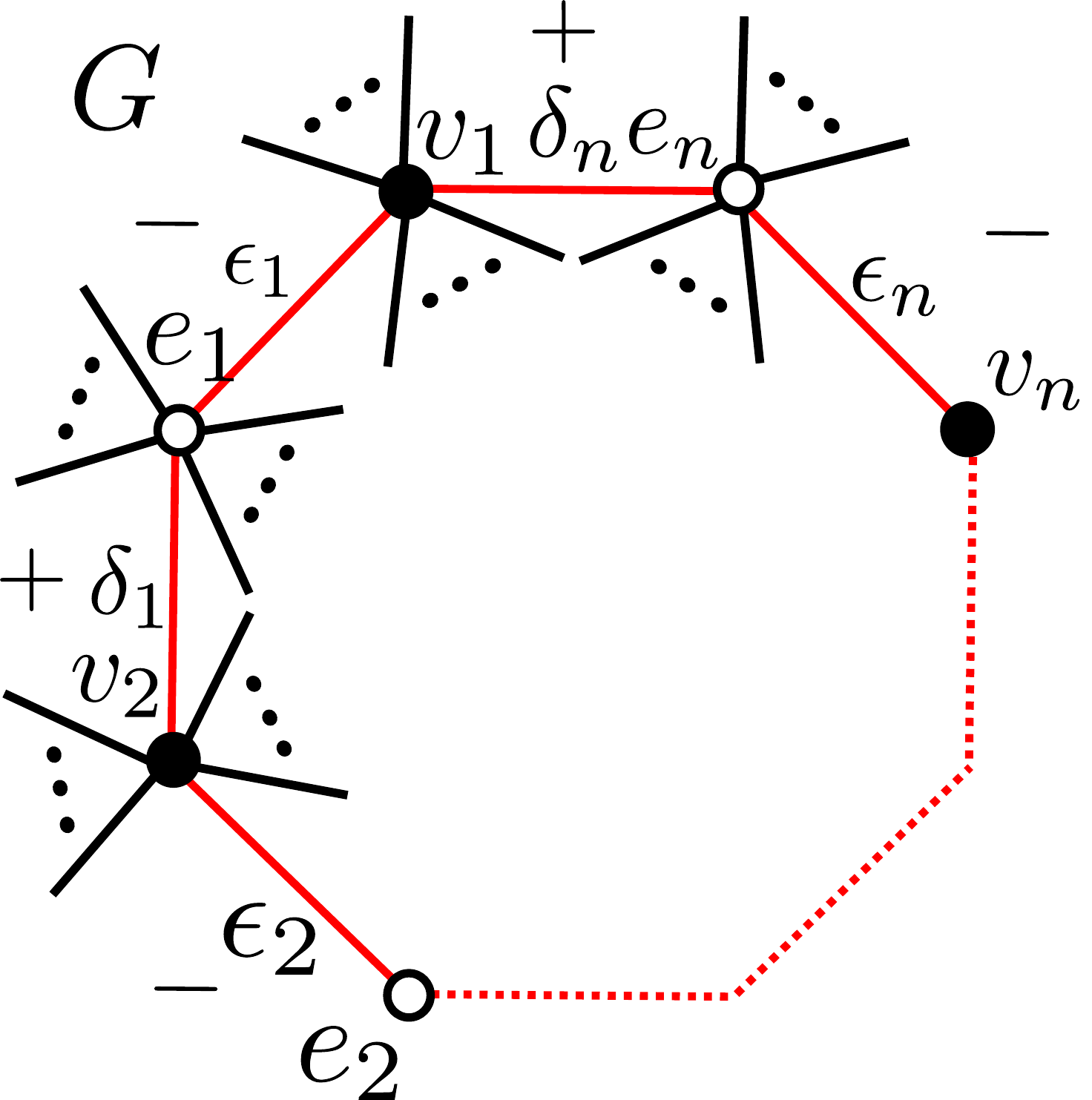}
\caption{An alternating cycle.}
\label{fig:altcycle}
\end{center}
\end{figure}

\begin{proof}[Proof of Theorem \ref{thm:0}]
It is sufficient to show that $I^+_G(x)=0$ in case the signed bipartite graph $G=(V,E,\mathcal{E})$ does not have negative edges except along the alternating cycle. If $G$ has a negative edge $\epsilon$ not on for the alternating cycle, we use the formula that $I^+_G(x)=I^+_{G+\epsilon}(x)-I^+_{G \setminus \epsilon}(x)$ (see Lemma \ref{lem:sumint}) recursively.

We denote by $v_1,e_1,v_2,e_2,\ldots ,v_n,e_n$ the vertices in the alternating cycle in $G$. Here $\epsilon_1=e_1v_1,\epsilon_2=e_2v_2,\ldots ,\epsilon_n=e_nv_n$ denote the negative edges and $\delta_{n+1}=e_1v_2,\delta_{n+2}=e_1v_2,\ldots ,\delta_{2n}=e_nv_1$ denote the positive edges in the alternating cycle. The other edges of $G$ are denoted by $\gamma_{2n+1},\ldots,\gamma_{|\mathcal{E}|}$, where $|\mathcal{E}|$ is the number of edges in $G$. The sums of standard generators $\bm{\epsilon}_1={\bf e}_1+{\bf v}_1,\bm{\epsilon}_2={\bf e}_2+{\bf v}_2,\ldots ,\bm{\epsilon}_n={\bf e}_n+{\bf v}_n,\bm{\delta}_{n+1}={\bf e}_1+{\bf v}_2,\bm{\delta}
_{n+2}={\bf e}_1+{\bf v}_2,\cdots,\bm{\delta}_{2n}={\bf e}_n+{\bf v}_1,\bm{\gamma}
_{2n+1},\ldots,\bm{\gamma}_{|\mathcal{E}|}$ are the point in $\mathbb{R}^{E}\oplus\mathbb{R}^{V}$ corresponding to the edges in $G$. Let $G_{\{i_1,\cdots,i_k\}}$ denote the bipartite graph obtained from $G$ by deleting the edges $\epsilon_{i_1},\cdots,\epsilon_{i_k}$ and forgetting sign. Then by definition,
\[
I^+_G(x)=\sum_{S \subset \{1,2,\cdots,n\}}(-1)^{|S|}I_{G_S}(x).
\]
Let $Q_{\{i_1,\cdots,i_k\}}$ denote the root polytope of the bipartite graph $G_{\{i_1,\cdots,i_k\}}$. So we have 
\[Q_\emptyset=\conv \{\bm{\epsilon}_1,\bm{\epsilon}_2,\cdots ,\bm{\epsilon}_n,\bm{\delta}_{n+1},\bm{\delta}
_{n+2},\cdots,\bm{\delta}_{2n},\bm{\gamma}
_{2n+1},\cdots,\bm{\gamma}_{|\mathcal{E}|}\},
\]
where $Q_\emptyset$ is the root polytope of the bipartite graph $G\setminus\emptyset$, which is the unsigned bipartite graph obtained from $G$ by forgetting signs.

We first show that $\displaystyle\bigcup_{k=1}^nQ_{\{k\}}=Q_\emptyset$. It is clear that $Q_{\{k\}}\subset Q_\emptyset$ for any $k$. So we have $\displaystyle\bigcup_{k=1}^nQ_{\{k\}}\subset Q_\emptyset$. Next we show that $\displaystyle\bigcup_{k=1}^nQ_{\{k\}}\supset Q_\emptyset$. We take ${\bf x} \in Q_\emptyset \subset \mathbb{R}^{E}\oplus \mathbb{R}^{V}$. We regard ${\bf x}$ as the weight system $\mathcal{W}$. We will denote by $\lambda_1,\lambda_2, \cdots, \lambda_n $ the weights of the edges $\epsilon_1, \epsilon_2, \cdots ,\epsilon_n$ in $\mathcal{W}$. Moreover, like in Example \ref{ex:cycchange}, we take $\lambda=\min\{\lambda_1,\lambda_2, \cdots, \lambda_n\}$ and apply cycle change. Thus, we get the weight system $\mathcal{W}'$ representing ${\bf x}$ so that there is an edge $\epsilon_k$ of which the weight in $\mathcal{W}'$ is $0$. Then we may regard the weight system as representing a point in the root polytope of the bipartite graph $G_{\{k\}}$. Because we found a $k$ such that ${\bf x} \in Q_{G_{\{k\}}}$, we have $\displaystyle\bigcup_{k=1}^nQ_{\{k\}}\supset Q_\emptyset$. Therefore $\displaystyle\bigcup_{k=1}^nQ_{\{k\}}=Q_\emptyset$.

For any set $A \subset \mathbb{R}^{n}$, let the function $[A]:\mathbb{R}^{n} \to \mathbb{R}$ be defined by 
\begin{equation*}
[A](x) = \begin{cases}
           1 & (x \in A) \\
           0 & (x \notin A)
         \end{cases}.
\end{equation*}
We call this function an indicator function.
Let $\Delta \subset \mathbb{R}^{|\mathcal{E}|}$ be the standard simplex
\[
\Delta = \left\{ (\lambda_1, \cdots, \lambda_{|\mathcal{E}|}) \relmiddle| \sum_{i=1}^{|\mathcal{E}|}\lambda_i=1 \mbox{ and }\lambda_i \ge 0 \mbox{ for }i=0,\cdots,|\mathcal{E}| \right\}.
\]
For $j=1,\cdots,n$, let $H_j \subset \Delta$ be defined by
\[
H_j = \left\{ (\lambda_1, \cdots, \lambda_{|\mathcal{E}|}) \relmiddle| \sum_{i=1}^{|\mathcal{E}|}\lambda_i=1 \mbox{ and }\lambda_i \ge 0 \mbox{ for }i=0,\cdots,|\mathcal{E}| \mbox{ and }\lambda_j=0 \right\}.
\]
Using the inclusion-exclusion formula, we write
\[
\left[ \bigcup_{j=1}^{n}H_j \right]=\sum_{\emptyset \neq J\subset \{1,\cdots,n\}}(-1)^{|J|-1}\left[H_J\right],\mbox{ where }H_J=\bigcap_{j \in J}H_j.
\]
Now we consider a linear transformation $T:\mathbb{R}^{|\mathcal{E}|}\to \mathbb{R}^{|E|+|V|}$ defined by 
\[
T(\lambda_1, \cdots, \lambda_{|\mathcal{E}|})=\sum_{i=1}^{n} \lambda_i\bm{\epsilon}_i+\sum_{i=n+1}^{2n}\lambda_i\bm{\delta}_i +\sum_{i=2n+1}^{|\mathcal{E}|}\lambda_i\bm{\gamma}_i.
\]

Then $T(\Delta)=Q_G$. Since linear transformations preserve linear dependencies among the indicator functions of polyhedra (see \cite[Theorem 3.1]{B}), applying $T$ to the inclusion-exclusion formula, we obtain
\[
\left[ T\left(\bigcup_{j=1}^{n}H_j\right) \right]=\sum_{\emptyset \neq J\subset \{1,\cdots,n\}}(-1)^{|J|-1}\left[T(H_J)\right].
\]
By the definition of $Q_J$, we have $\left[T(H_J)\right]=Q_J$ for all $J\subset\{1,\cdots n\}$. On the other hand, from what has already been proved, we have
\[
\left[ T\left(\bigcup_{j=1}^{n}H_j\right) \right]=\left[ \bigcup_{j=1}^{n}T(H_j) \right]=\left[ \bigcup_{j=1}^{n}Q_{\{j\}}\right]=[Q_{\emptyset}].
\]
For these reasons, we get
\begin{gather*}
[Q_{\emptyset}]=\sum_{\emptyset \neq J\subset \{1,\cdots,n\}}(-1)^{|J|-1}\left[Q_J\right],\mbox{ in other words }
\sum_{J\subset \{1,\cdots,n\}}(-1)^{|J|}\left[Q_J\right]=0.
\end{gather*}
By the definition of the Ehrhart polynomial, for any $s\in\mathbb{N}$, we have
\[
\sum_{J\subset \{1,\cdots,n\}}(-1)^{|J|}|(s\cdot Q_J) \cap (\mathbb{Z}^{E}\oplus\mathbb{Z}^{V})|=0.
\]
By Theorem \ref{thm:signserint},
\begin{eqnarray*}
\frac{I^+_G(x)}{(1-x)^{|E|+|V|-1}}
&=&\ehr^+_{G}(x)\\
&=&\sum_{\mathcal{S} \subseteq \mathcal{E}_{-}(G)}(-1)^{|\mathcal{S}|}\ehr_{Q_{G \setminus \mathcal{S}}}(x)\\
&=&\sum_{J\subset \{1,\cdots,n\}}(-1)^{|J|}\ehr_{Q_{J}}(x)\\
&=&\sum_{J\subset \{1,\cdots,n\}}(-1)^{|J|}\left(1+\sum_{s\in \mathbb{N}}\varepsilon_{Q_J}(s) x^s\right)\\
&=&\sum_{J\subset \{1,\cdots,n\}}(-1)^{|J|}\left(1+\sum_{s\in \mathbb{N}}|(s\cdot Q_J) \cap (\mathbb{Z}^{E}\oplus\mathbb{Z}^{V})| x^s\right)\\
&=&\sum_{J\subset \{1,\cdots,n\}}(-1)^{|J|}+\sum_{s\in \mathbb{N}}\left(\sum_{J\subset \{1,\cdots,n\}}(-1)^{|J|}|(s\cdot Q_J) \cap (\mathbb{Z}^{E}\oplus\mathbb{Z}^{V})|\right) x^s\\
&=&0.
\end{eqnarray*}
Therefore, we have $I^+_G=0$.
\end{proof}

\subsection{A recusion formula}\label{sec:rec}
In section \ref{sec:rec} and \ref{sec:morton}, we discuss applications of Theorem \ref{thm:0}. Theorem \ref{thm:0} gives the identity for the original interior polynomial stated as Corollary \ref{cor:recursion}.
\begin{proof}[Proof of Corollary \ref{cor:recursion}]
We assume that the unsigned bipartite graph $G$ is not a tree. We find the cycle in $G$ and label the edges of the cycle $\epsilon_1,\delta_1,\epsilon_2,\delta_2,\cdots,\epsilon_n,\delta_n$. We regard the signs of $\epsilon_1,\epsilon_2,\cdots,\epsilon_n$ as negative and the signs of the other edges as positive. Then we construct the signed bipartite graph $G'$ containing the alternating cycle and the signed interior polynomial of $G'$ is $0$. By the definition of the signed interior polynomial, we have
\[
\sum_{\mathcal{S}\subset\{\epsilon_1,\epsilon_2,\cdots,\epsilon_n\}}(-1)^{|\mathcal{S}|}I'_{G'\setminus\mathcal{S}}(x)=0.
\]
The unsigned bipartite graph $G'\setminus\mathcal{S}$ is the bipartite graph $G\setminus \mathcal{S}$, which is obtained from $G$ by deleting edges in $\mathcal{S}$. Moreover, if $\mathcal{S}=\emptyset$, the unsigned bipartite graph $G'\setminus\mathcal{S}$ is the original unsigned bipartite graph $G$. We have
\[
I'_G(x)=\sum_{\emptyset\neq\mathcal{S}\subset\{\epsilon_1,\epsilon_2,\cdots,\epsilon_n\}}(-1)^{|\mathcal{S}|-1}I'_{G\setminus\mathcal{S}}(x).
\]
This completes the proof.
\end{proof}
This identity says that the original interior polynomial $I_G(x)$ of the bipartite graph $G$ is obtained from the interior polynomials of bipartite graphs that result from deleting some edges in a cycle. This is one possible counterpart of the deletion-contraction relation of the Tutte polynomial. Moreover, in the bipartite graph $G\setminus\mathcal{S}$, which is obtained from $G$ by deleting some edges $\mathcal{S}\neq\emptyset$ in the cycle, at least one cycle of $G$ disappears. Thus we may apply the identity repeatedly and get the original interior polynomial by the interior polynomials of some spanning forests. The interior polynomial of a forest $F$ which has $c(F)$ components is $(1-x)^{c(F)-1}$. Therefore we get the interior polynomial without computing activities as in the original definition.

\begin{example}
Let $G$ be the bipartite graph $K_{23}$ shown in Figure \ref{k23}. We take edges in a cycle in $G$, and delete some of the edges. We repeat the operation until the graph becomes a forest, as in Figure \ref{fig:computation}. We compute the interior polynomials of trees with suitable signs. We sum these polynomials, and obtain $I_G=-1+4-2(1-x)=1+2x$.

\begin{figure}[htbp]
\includegraphics[width=8.5cm]{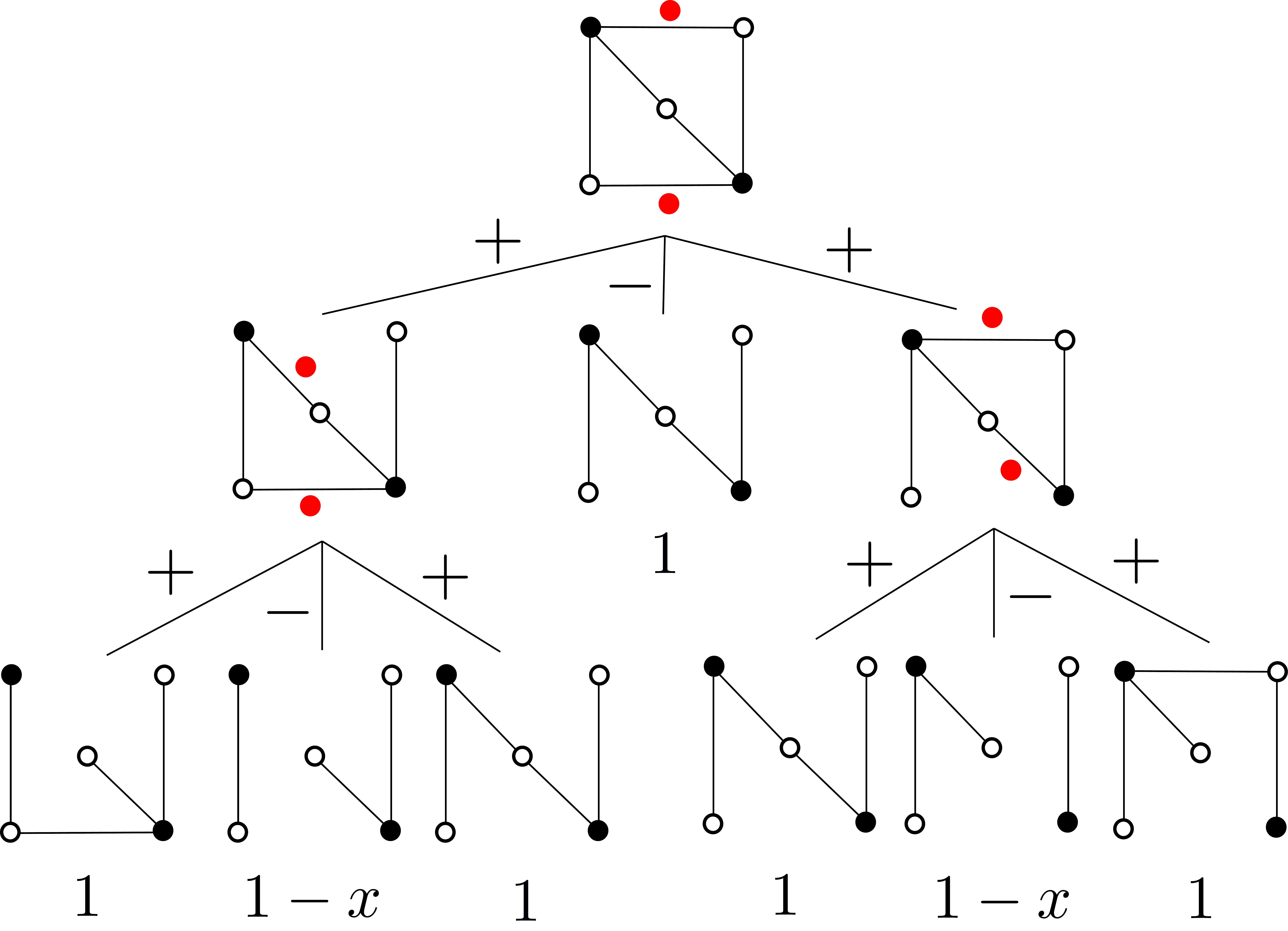}
\caption{A computation tree of $G$.}\label{fig:computation}
\end{figure}
\end{example}

\subsection{The vanishing formula and Morton's inequality}\label{sec:morton}
For any oriented link diagram $D$, We get the Seifert graph $G=(V,E,\mathcal{E})$ of $D$, which is signed bipartite graph. Morton's inequality implies that the maximal $z$-exponent of the HOMFLY polynomial $P_{D}$ is less than or equal to $|\mathcal{E}|-(|E|+|V|)+1$. If the maximal $z$-exponent of the HOMFLY polynomial $P_{D}(v,z)$ is less than $|\mathcal{E}|-(|E|+|V|)+1$, we say that Morton's inequality is not sharp. Theorem \ref{thm:any} implies that the part with $z$-exponent $|\mathcal{E}|-(|E|+|V|)+1$ of the HOMFLY polynomial $P_{D}(v,z)$ is obtained from the signed interior polynomial $I^+_G(x)$. So when oriented link diagram has the Seifert graph $G$ with $I^+_G(x)=0$, Morton's inequality is not sharp. Therefore we get the following corollary by Theorem \ref{thm:0}. 

\begin{cor}\label{cor:sharp}
If the Seifert graph of an oriented link diagram contains an alternating cycle, then Morton's inequality is not sharp.
\end{cor}

When the Seifert graph consists of blocks in which the signs are the same, the link diagram with this Seifert graph is a homogeneous diagram. In \cite{C}, Cromwell showed that, for any homogeneous diagram, Morton's inequality is sharp. Of course, homogeneous diagrams have no alternating cycles. So we can expect that if Morton's inequality is sharp, the Seifert graph of the oriented link diagram has an alternating cycle, but this is wrong.

\begin{example}
The following minimal, special diagram $D$ of $15_{100154}$ (see \cite{diagram}) is an example which shows that the maximal $z$-exponent of the HOMFLY polynomial dose not have to agree with $c(D)-s(D)+1=15-8+1=8$, even though the graph has no alternating cycles.
\[
\begin{array}{llllll}
P_{15_{100154}}(v,z)=
&                &+6v^{-2}z^6&              &+\,\,\,v^2z^6&              \\
&+\,\,\,v^{-4}z^4&-5v^{-2}z^4&+\,\,\,6v^0z^4&+     4v^2z^4&-\,\,\,v^4z^4 \\
&+     2v^{-4}z^2&-9v^{-2}z^2&+     10v^0z^2&+     4v^2z^2&-     3v^4z^2 \\
&+\,\,\,v^{-4}z^0&-5v^{-2}z^0&+\,\,\,6v^0z^0&+      v^2z^0&-     2v^4z^0. 
\end{array}
\]

\begin{figure}[htbp]
\begin{tabular}{cc}\hspace{-20pt}
\begin{minipage}{0.5\hsize}
\centering
\includegraphics[width=4.5cm]{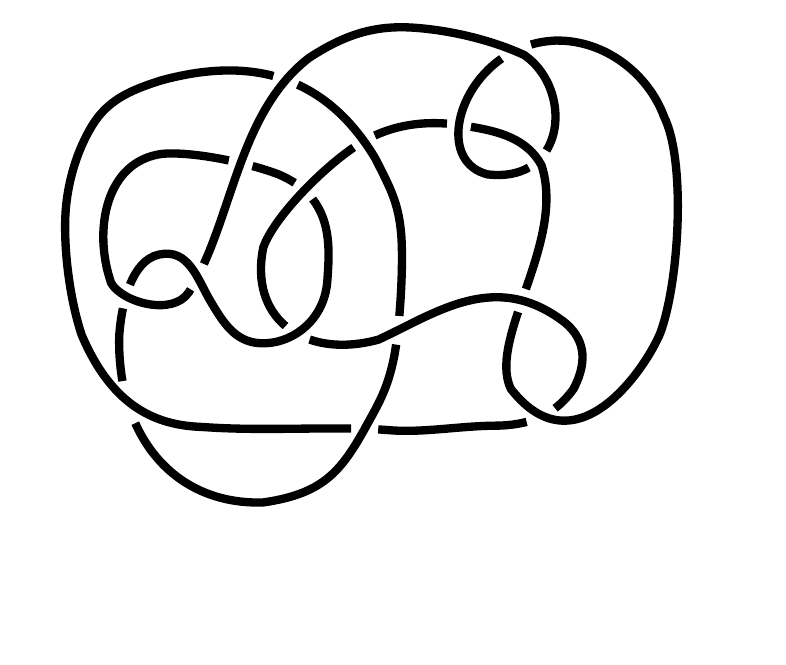}
\caption{A diagram of $15_{100154}$.}

\end{minipage}
\begin{minipage}{0.6\hsize}
\centering
\includegraphics[width=4.5cm]{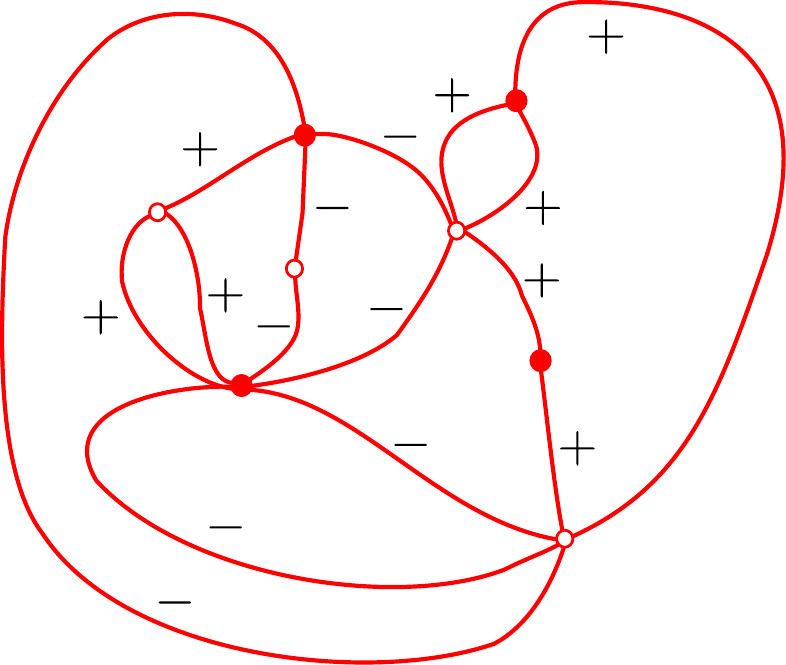}
\caption{The Seifert graph of $15_{100154}$.}

\end{minipage}
\end{tabular}
\end{figure}
\end{example}

\end{document}